\documentclass[11pt,a4paper]{article}

\usepackage{amssymb}
\usepackage{amsmath}
\usepackage{amsthm}
\usepackage[showonlyrefs]{mathtools}
\mathtoolsset{showonlyrefs}
\usepackage[breaklinks=true,colorlinks, linkcolor=blue, anchorcolor=blue, citecolor=blue]{hyperref}

\usepackage[dvips]{graphicx}
\usepackage{psfrag}
\usepackage{amscd}
\usepackage{enumerate}

\numberwithin{equation}{section}

\newcommand{\comm}[1]{}

\def\diam{\operatorname{diam}}
\def\dist{\operatorname{dist}}

\def\crit{\operatorname{Crit}}

\def\Exp{\operatorname{Exp}}

\def\ti{\tilde}
\def\wt{\widetilde}

\def\({\left(}
\def\){\right)}
\def\oli{\overline}
\def\uli{\underline}

\def\raw{\rightarrow}

\def\no={\neq}
\def\sm{\setminus}

\def\C{{\mathbb C}}

\def\BB{{\mathcal B}}

\def\EE{{\mathcal E}}
\def\FF{{\mathcal F}}

\def\JJ{{\mathcal J}}

\def\MM{{\mathcal M}}
\def\NN{{\mathcal N}}

\def\QQ{{\mathcal Q}}

\def\al{\alpha}
\def\be{\beta}
\def\ga{\gamma}
\def\de{\delta}

\def\vep{\varepsilon}

\def\io{\iota}
\def\ka{\kappa}

\def\om{\omega}

\def\De{\Delta}

\theoremstyle{plain}

\newtheorem{Thm}{Theorem}[section]
\newtheorem{Prop}[Thm]{Proposition}
\newtheorem{Lem}[Thm]{Lemma}

\newtheorem{Cor}[Thm]{Corollary}

\newtheorem{Def}[Thm]{Definition}

\newtheorem{Rem}[Thm]{Remark}

\begin{document}

\title{Collet--Eckmann maps in the unicritical family}

\author{Magnus Aspenberg, Mats Bylund and Weiwei Cui}

\date{}

\maketitle

\begin{abstract}
 In this paper we show that Collet--Eckmann parameters in the unicritical family $f_c(z) = z^d + c$, where $d$ is at least $2$, are Lebesgue density points of the complement of the Mandelbrot set.%, i.e. the connectedness locus.
% Combining this with a result by S. Smirnov, the Mandelbrot set is an example of a simply connected compact set, whose interior is% only visible if one apporaches it along a ray of angles in a set of Hausdorff dimension zero. 
\end{abstract}


 \section{Introduction}

Consider the unicritical family
$$f_{c}(z)=z^d+c,$$
where $d\geq 2$ is an integer and $c\in\C$ is a parameter. 
This family has been a central object of study in complex dynamics for a long time. One of the main objectives is to understand it from a topological and metric point of view. On the topological side, the famous Fatou conjecture, stating that the set of hyperbolic parameters is open and dense, was settled for real parameters $c$, for $d=2$, by J.~Graczyk and G.~\'Swi\c{a}tek~\cite{GSW}, and independently M.~Lyubich~\cite{ML2}. It was settled in full generality, that is, for real polynomials, by O.~Kozlovski, W.~Shen and S.~van~Strien~\cite{SSK}. On the metric side, seminal works by M.~Lyubich~\cite{ML} and A.~Avila and G.~Moreira~\cite{AM} together show that almost all real quadratic polynomials are either hyperbolic (nice robust maps with an attracting cycle) or expanding in a strong sense (i.e. satisfy the Collet--Eckmann condition, see below).

The Collet--Eckmann condition arose early as a natural condition for the presence of chaotic behaviour and the existence of absolutely continuous invariant measures (acim)~\cite{Collet-Eckmann1, Collet-Eckmann2, Collet-Eckmann3}. The pioneering work of M.~Jakobson~\cite{MJ} showed that parameters in the real quadratic family admitting an acim have positive Lebesgue measure. M.~Rees proved in~\cite{MR} a corresponding celebrated result for rational maps. Later, M.~Benedicks and L.~Carleson proved in the fundamental papers~\cite{BC1, BC2} that the set of CE-parameters in the real quadratic family has positive Lebesgue measure. In the rational setting, a similar result was proven by the first author in~\cite{MA-Z}. 

Let us recall some definitions. A function $f_c$ is called \emph{Collet--Eckmann} if there exist $C>0$ and $\gamma>0$ such that
\begin{equation}\label{cedef}
\vert Df_{c}^{n}(c)\vert \geq Ce^{\gamma n}~\,~\text{for all}~\, n > 0.
\end{equation}
A map $f_c$ in the unicritical family is {\em hyperbolic} if $\om(0) \cap \JJ(f_c) = \emptyset$, where $\JJ(f_c)$ is the Julia set of $f_c$. For $d \geq 2$, the set $\mathcal{M}_d$ is defined as the set of parameters $c$ such that $\JJ(f_c)$ is connected (or equivalently, the set of parameters $c$ for which the critical orbit $\{f^{n}_{c}(0)\}$ is bounded). When $d=2$, the set $\mathcal{M}_2 = \MM$ is the well known \emph{Mandelbrot set}. (We shall also call $\MM_d$ a Mandelbrot set.)

Our paper is concerned with (complex) perturbations of CE-maps in the unicritical family. In~\cite{RL2}, J.~Rivera-Letelier proved that Misiurewicz maps (i.e. critically non-recurrent maps) are Lebesgue density points of $\C \sm \MM_d$. It was generalised by the first author in~\cite{MA5} to rational maps. However, even if Misiurewicz maps behave quite nicely under perturbations, they are rare in the parameter space~\cite{MA3}. 

In a series of papers~\cite{GSW-Harmonic, GSW-Fine, GSW-Struc}, J.~Graczyk and G.~\'{S}wi\c{a}tek studied the geometry of the Julia set around typical maps with respect to harmonic measure on $\partial \MM_d$. Among other things, they show that almost all maps with respect to harmonic measure are critically slowly recurrent CE-maps and density points of $\C \sm \MM$ (we refer to slow recurrence in the sense of G.~Levin, F.~Przytycki and W.~Shen~\cite{Le-Pr-Sh}). 

In the other direction, A.~Avila, M.~Lyubich and W.~Shen proved in~\cite{ALS} that the Lebesgue density of $\MM_d$ at at most finitely renormalizable maps in the unicritical family, is strictly less than one (Theorem 1.4 and Remark 6.1 in~\cite{ALS}). Since Collet--Eckmann maps are at most finitely renormalizable, the result applies to them and consequently CE-maps cannot be density points of $\MM_d$.

Our main theorem complements the above results as follows.

\begin{Thm}\label{mainthm}
In the unicritical family, Collet--Eckmann maps are Lebesgue density points of $\C \sm \MM_d$.
\end{Thm}

%We first make a remark on the result.

\begin{Rem}
One could ask if our result extends to the setting of at most finitely renormalizable maps, i.e., those considered by A.~Avila, M.~Lyubich and W.~Shen. We remark that this is not the case, as was pointed out in~\cite{GSW-Sim} by J.~Graczyk and G.~\'{S}wi\c{a}tek .

In the same paper just mentioned, J.~Graczyk and G.~\'{S}wi\c{a}tek showed that at most finitely renormalizable maps are weak density points of $\C \sm \MM_d$ (again, this applies to CE-maps). Thus our theorem strengthens their result for CE-maps.
\end{Rem}

\comm{
Similar to the Misiurewicz case (which is easier), the goal is to get some kind of similarity between phase and parameter space. One cannot expect a complete (conformal) similarity between them since the Mandelbrot set contains an interior and the Julia set for a CE-map is a dendrite. However, at small scales this similarity becomes asymptotically conformal. The result is obtained by introducing a powerful tool called the similarity map, which transfers the geometry around the critical value $c$ between the Julia set and the Mandelbrot set $\MM_d$ in an intriguing way.
In particular, it is shown in~\cite{GSW-Struc} that typical Collet--Eckmann maps, with respect to harmonic measure, are density points of $\C \sm \MM_d$.
}

The slow recurrence condition is also imposed on the Collet--Eckmann maps in~\cite{ABC-1}, where we proved a corresponding result for rational maps. We also mention an unpublished work by M.~Benedicks and J.~Graczyk, where they prove corresponding results for one-dimensional families (with some regularity) around typical CE-maps with respect to harmonic measure in $ \partial \MM_d$. For more general CE-maps, the dynamical constants can be drastically different. This yields quite wild dynamics, and crucial parts of the traditional methods break down. 

\comm{
\begin{Rem}
 It will be clear from the proof that the methods used in this paper can be applied to get the same conclusion as in Theorem~\ref{mainthm} for rational CE-maps with precisely one critical point in the Julia set, namely that such CE-maps are Lebesgue density points of hyperbolic maps. We also expect that the methods extend to every CE-map for which the Julia set is not the whole sphere, i.e. a generalisation of~\cite{ABC-1}, for spaces of rational maps where the orders of the critical points are fixed. 
\end{Rem}

\begin{Rem}
Combining the previous remark with fact that every CE-map, with only simple critical points, and for which the Julia set is the whole sphere, is a Lebesgue density point of CE-maps for which, again, the Julia set is the whole sphere (see~\cite{Lefevre-thesis}, and earlier~\cite{MA7}), we get that almost all CE-maps of degree $2$ have their Julia set equal to the whole sphere (maps with higher order critical points have zero Lebesgue measure). We expect this to generalise to all degrees. 
\end{Rem}
}

Theorem~\ref{mainthm} is also interesting from a complex analytic viewpoint. Indeed, in~\cite{Smirnov-CE}, S. Smirnov proved that for a set $E$ of angles whose complement has Hausdorff dimension $0$, the external rays with angles in $E$ will land on a Collet--Eckmann parameter on the boundary of the Mandelbrot set. As a consequence, by Theorem~\ref{mainthm}, the Mandelbrot set is an example of a simply connected compact set where a large set of boundary points, in the above sense, has Lebesgue density zero. So from the outside, the Mandelbrot set looks like a one-dimensional object; the interior is not detectable from the outside. 
%\textbf{MAKE THIS MORE CLEAR AND ALSO NOT SO SUPERFICIAL}

%%%%%%%%%%% The below theorem is now taken away. 
\comm{
Recall that a rational map satisfies the Collet--Eckmann condition if the map has no parabolic cycles and~\eqref{cedef} is satisfied for each critical point in the Julia set (usually one requires the map to be non-hyperbolic). Let $\mathcal{R}_d$ be the space of rational maps of degree $d\geq 2$.

\begin{Thm}\label{ratdensity}
Any rational Collet--Eckmann map of degree $d$ at least two with only simple critical points and with precisely one critical point in the Julia set is a Lebesgue density point of hyperbolic maps in $\mathcal{R}_d$.
\end{Thm}

This also partially improves our previous result in~\cite{ABC-1} where we imposed the additional assumption of slow recurrence of the critical points in the Julia set.

\begin{Rem}
In case that rational maps have critical points of higher multiplicities, one can also get a corresponding result but in a subspace of $\mathcal{R}_d$. This subspace was defined by G.~Levin~\cite{Levin-book}.
\end{Rem}

}

\bigskip

\noindent\emph{Ideas of the proof.} We trace the critical orbits for all parameters close to a Collet--Eckmann map and use the parameter exclusion technique developed by M.~Benedicks and L.~Carleson in the seminal works~\cite{BC1, BC2}. However, since we start with arbitrary Collet--Eckmann maps, the dynamical constants, e.g. the recurrence rate, can be quite wild in general (deep central returns) and traditional methods break down. This difficulty will be overcome by introducing the idea of \emph{promotion}. Roughly speaking, promotion means that almost all parameters will reach a condition arbitrarily close to the slow recurrence condition. This technique is one of the core novelties and fundamental for the result of the paper.

\bigskip

\noindent\emph{Acknowledgement.} Partially this work was written when all the authors were participating in the program ``Two Dimensional Maps'' at Mittag-Leffler Institute during the spring 2023. We thank the organisers for providing a nice working place. We are also grateful to Michael Benedicks for stimulating discussions on the topic of this paper. The first author gratefully acknowledges partial funding by the Swedish Research Council. The second author gratefully acknowledges support from the Knut and Alice Wallenberg Foundation. 
The third author was partially supported by NSF of China (No. 12401105, 12522108), Qingdao Natural Science Foundation (No. 24-4-4-zrjj-8-jch), Shandong Provincial Natural Science Fund for Excellent Young Scientists Program (Overseas) (No. 2025HWYQ-021) and also a grant from Vergstiftelsen during his stay in Sweden.

\section{Preliminaries}

We collect some preliminary notions and notations in this section. For some general notions such as the Fatou and Julia sets, and other related notions in complex dynamics, we refer to~\cite{CG,Milnor-book}. We will always use the Euclidean metric in this paper and thus distances and diameters are measured with respect to this metric. The notation $\dist(A, B)$ means the Hausdorff distance between two sets $A$ and $B$, and $\diam(A)$ denotes the diameter of $A$. Also, $m(A)$ is the two-dimensional Lebesgue measure (i.e., area) of the set $A$.

%Conjecturally, almost all parameters satisfy the (critically) slow recurrent condition, defined as follows.

%\begin{Def}[Slow recurrence]\label{SL}
%A point $z$ is said to be \emph{slowly recurrent} if for any $\alpha>0$, there exists $K>0$ such that
%\begin{equation} \label{slow-recurrence}
%\dist(f^{n}(z), \jrit(f))\geq Ke^{-\alpha n}~\,~\text{for all}~\,n\geq 0.
%\end{equation}
%Moreover, we say that $f$ is slowly recurrent if every point in $\jrit(f)$ is slowly recurrent.
% \end{Def}

Throughout the paper we fix $d\geq 2$. Let $f_{c_0}$ be a Collet--Eckmann unicritical polynomial of degree $d$. The only critical point is $0$ with the corresponding critical value being $c_0$. Let $\QQ=\QQ(c_0,\varepsilon)$ be a square of side length $\varepsilon$ in the parameter plane centred at the parameter $c_0$ with $\varepsilon>0$ to be sufficiently small (cf. Lemma~\ref{levin} and Lemma~\ref{oel}). For a parameter $c\in\QQ$ the corresponding map is denoted, as usual, by $f_c$. We study the evolution of the critical orbit for nearby parameters. Let
\[
\xi_n(c)=f_{c}^{n}(0)\quad\text{for}\quad n\geq 0.
\]
We will need to fix some neighbourhood for the critical point and study recurrence of $\xi_n(c)$ to this neighbourhood for $c$ close to $c_0$ (cf. Lemma~\ref{oel}). Let $0<\Delta'<\Delta$ be large integers to be determined later and put
\begin{equation}\label{uu'}
U=D(0, e^{-\Delta})\quad\text{and}\quad U'=D(0, e^{-\Delta'}).
\end{equation}
In the following, we also put $\delta=e^{-\Delta}$ and $\delta'=e^{-\Delta'}$. 

Starting with a Collet--Eckmann map, we have exponential growth of derivative along the critical orbit. This expansion will be inherited by nearby parameters with possibly slightly slower growth at the initial piece of their critical orbits. This is ensured by the transversality at Collet--Eckmann parameters (see Lemma~\ref{levin}). To gain growth for later orbits, we divide the orbit into pieces: \emph{bound periods} and \emph{free periods}. During the bound periods, the orbits shadow the initial iterates of its critical orbit and expansion is ensured (see Section 3.1). During free periods, the orbits stay away from the critical point and hence gain exponential growth of derivatives (Lemma~\ref{oel}) similar as the classical Ma\~n\'e theorem (Lemma~\ref{eafcp}).

In the rest of this section we present necessary notions which are crucial in our later analysis and also some other related results concerning recurrence rates of Collet--Eckmann maps, expansions during bound periods and finally the transversality result.

\subsection{Recurrence, bound periods and partitions}

With the critical neighbourhoods defined in~\eqref{uu'}, we say that $\xi_n(c)$ is a \emph{return} if $\xi_n(c)\in U'$. It is a {\em pseudo-return} if $\xi_n(c)\in U' \sm U$. Sometimes we specify and say that, for instance, $\xi_n(c)$ is a return into $U$. If $\xi_n(c)$ is a return, the following orbit will be shadowed by its early piece of the critical orbit of $f_{c}$ until they separate in the following sense. The number $n$ in this case is called a return time. 

\begin{Def}\label{pwbound}
Let $\be > 0$. Let $\xi_n(c)$ be a return to $U'$. The bound period for this return is defined as the indices $j>0$ such that
\[
|\xi_{n+j}(c)-\xi_j(c)|\leq e^{-\beta j}|\xi_j(c)|.
\]
The largest number $p$ with the above inequality is called the length of the bound period (for this return).
\end{Def}

The number $\be > 0$ should be thought of as a comparatively small number. There will be some conditions imposed on $\be$, e.g. in Section~\ref{large-dev}. 

After the bound period ends, the two critical orbits will then separate and stay away from critical points until the next return happens. A return might be deep in the sense that $\xi_n(c)$ is quite close to $0$. To achieve our goal, we will need to disregard parameters with very deep returns. For later purpose we will define $\alpha_n(c)$ for a parameter $c$ as follows:
\[
|\xi_n(c)|= e^{-\alpha_n(c) n} \quad\text{for all}\quad n.
\]
It is clear that if one is considering a Misiurewicz or semi-hyperbolic parameter, then $|\xi_n(c)|$ has a uniform lower bound. For general Collet--Eckmann parameters, this need not be true. The map $f_c$ is (critically) {\em slowly recurrent} (see e.g.~\cite{Le-Pr-Sh}) if $$\lim\limits_{n \raw \infty} \al_n(c) = 0.$$

Recall that the lower and upper Lyapunov exponents at the critical value $c$ are defined respectively as
\[
\underline{\gamma}(c):=\liminf_{n\to\infty}\frac{1}{n}\log|Df_{c}^{n}(c)| \quad\text{and}\quad\overline{\gamma}(c):=\limsup_{n\to\infty}\frac{1}{n}\log|Df_{c}^{n}(c)|
\]
The Collet--Eckmann condition simply means that $\underline{\gamma}(c_0)>0$ and that $f$ is not hyperbolic. Let $\gamma_n(c)$ be defined such that
$$|Df_{c}^{n}(c)|=e^{\gamma_n(c) n} \quad\text{for all~}~n,$$
where $\gamma_n(c)$ depends on the parameter $c$. Note also that $|Df_{c}^{n}(z)|$ grows at most exponentially in $n$ for any $c$ and any $z$. So $\overline{\gamma}(c)$ indeed has an upper bound. Both $\al_j=\al_j(c)$ and $\ga_j=\ga_j(c)$ depend on the parameter, but we usually suppress it. 
Suppose that the starting map $f=f_{c_0}$ satisfies, for some $C_0 > 0$ and $\ga > 0$, 
\[
|Df_{c_0}^n(c_0)| \geq C_0 e^{\ga n} \quad\text{ for all ~}\, n > 0.
 \]
In other words, $0 < \ga \leq \uli{\ga}(c_0)$. We now make the following definition. Let $\ti{N} > 0$ be a number which satisfies $\ga_k(c_0) = \frac{1}{k} \log |Df_{c_0}^{k}(c_0)| \geq (9/10) \uli{\ga}(c_0) > 0$, for all $k \geq \ti{N}$, and that $\ga_{\ti{N}}(c) = \frac{1}{\ti{N}} \log |Df_{c}^{\ti{N}}(c)| > (9/10) \uli{\ga}(c_0)$ for all $c \in \QQ$. Moreover, we also require that all bound periods for return times $n$ considered in the paper are larger than $\ti{N}$ (so $n$ is typically much larger than $\ti{N}$). This can be achieved by choosing $\de'$ small enough (see the Definition~\ref{pwbound}). 
\begin{Def}\label{EEga}
Let $\gamma>0$. We say that $c \in \EE_n(\ga)$ if $|Df_c^k(c)| \geq e^{\ga k}$ for all $\ti{N} \leq k \leq n$. 
\end{Def}
Note that we drop the constant $C_0$ in front, which appears in the starting function $f_{c_0}$. So in principle, $c_0 \in \EE_n(\uli{\ga}(c_0)- \vep)$, for very small $\vep > 0$. 
In our analysis later, we will specify the constant $\gamma$ that suits our needs.
%Moreover, we will also specify what is meant by sufficiently large $k$. 

Using the chain rule
\begin{equation}\label{initial}
\begin{aligned}
\left|Df_{c}^{n}(c)  \right|=\left|Df_{c}^{n-1}(c) \right|\,\left|Df_{c}(f_{c}^{n-1}(c)) \right|&=\left|Df_{c}^{n-1}(c) \right| \left|Df_{c}(\xi_n(c)) \right|\\
&=\left|Df_{c}^{n-1}(c) \right| d |\xi_n(c)|^{d-1},
\end{aligned}
\end{equation}
we get that
\begin{equation}\label{alpha_gamma_relation}
e^{\gamma_n n}=e^{\gamma_{n-1}(n-1)}\,d\,e^{-\alpha_n n(d-1)}.
\end{equation}
The following simple inequalities, for $n$ large enough,
\begin{equation}\label{alpha_gamma_1}
\gamma_n \geq \gamma_{n-1} - (d-1)\alpha_n - \frac{C}{n}
\end{equation}
and
\begin{equation}\label{alpha_gamma_2}
\gamma_{n-1} \geq \gamma_n + (d-1)\alpha_n - \frac{C}{n}
\end{equation}
will be used throughout the paper. Here $C$ can be chosen to be $\log d$.

%We also put 
%$$\underline{\gamma}(n)=\inf_{k\leq n}\frac{1}{k}\log|Df_{c}^{k}(c)| \,~\text{and}\,~\bar{\gamma}(n)=\sup_{k\leq n}\frac{1}{k}\log|Df_{c}^{k}(c)| $$
%For this aim we define the following \emph{evolving recurrence rate}.

Our first observation is that one still has some control over the recurrence of the critical point for a Collet--Eckmann map.
\begin{Lem}[Initial recurrence rate]\label{initial-lemma}
Let $f_{c}$ be a Collet--Eckmann map. Then there exist $K>0$ and $\alpha>0$ such that
$$\left|\xi_n(c) \right|\geq Ke^{-\alpha n} ~\,~\text{for all}~\,n\geq 0.$$
\end{Lem}

\begin{proof}
From the definition of the upper and lower Lyapunov exponents, given any $\epsilon > 0$ we can find $N \geq 0$ such that, with $\uli{\ga} = \uli{\gamma}(c)$ and $\oli{\gamma} = \oli{\gamma}(c)$, 
\[
e^{(\underline{\gamma} - \epsilon)n} \leq \vert Df_c^n(c) \vert \leq e^{(\overline{\gamma} + \epsilon)n} \quad (n \geq N).
\]
By~\eqref{initial}, we have
$$d |\xi_n(c)|^{d-1}=\frac{\left|Df_{c}^{n}(c)  \right|}{\left|Df_{c}^{n-1}(c) \right|}\geq \frac{e^{(\underline{\gamma}-\epsilon)n}}{e^{(\overline{\gamma}+\epsilon)(n-1)}} = e^{\overline{\gamma} + \epsilon} e^{-(\overline{\gamma}-\underline{\gamma} + 2\epsilon)n} \quad(n\geq N).$$
Therefore, there exists a constant $K > 0$ such that
$$|\xi_n(c)|\geq K e^{-\alpha n} \quad (n \geq 1),$$
where $\alpha=(\overline{\gamma}-\underline{\gamma} + 2\epsilon)/(d-1)$.
\end{proof}
By Lemma~\ref{initial-lemma}, for a Collet--Eckmann map the numbers $\alpha_n$ cannot tend to $\infty$, so the critical orbits cannot return too close to the critical point. 
\begin{Rem} \label{gamma-ulioli}
 We notice from the above lemma that if the Lyapunov exponent exists, i.e. if $\underline{\gamma} = \overline{\gamma}$, then the map satisfies the so-called slow recurrence condition. In this situation Theorem~\ref{mainthm} follows from~\cite{ABC-1}. (Compare also with~\cite{GSW-Fine,GSW-Struc}.) In what follows, we therefore assume that $\underline{\gamma} < \overline{\gamma}$.

% With some abuse of notations, we will keep the notation $\underline{\gamma}$ for $\underline{\gamma} - \epsilon$, and %$\overline{\gamma}$ for $\overline{\gamma} + \epsilon$, keeping in mind that $\epsilon$ can be taken arbitrarily small by considering %high iterations. So, in particular, we will assume that $\al_n(c) < \oli{\ga}(c) - \uli{\ga}(c)$, when $n \geq \ti{N}$, if $f_c$ is a %CE-map.
 
Actually, we take the upper exponent $\oli{\ga}$ to be a trivial upper bound for all exponents in the cube $\QQ$ at all times; i.e. $$\oli{\ga} = \sup\limits_{|z| \leq 2, c \in \QQ} \log |Df_c(z)|,$$
so that we have $\ga_n(c) \leq \oli{\ga}$, for all $n > 0$ and all $c \in \QQ$. Also, in order to have some space to play with when we change the parameter, we define
\[
\uli{\gamma} = \frac{9}{10} \uli{\ga}(c_0).
\]
The number $\uli{\gamma}$ will stand as a kind of lower Lyapunov exponent, at least up to the start-up time $N$, described in Lemma~\ref{startup}. Up until this time, we have $\al_n(c) < \oli{\ga} - \uli{\ga}$, ($\ti{N} \leq n \leq N$), for $c \in \QQ$. 
\end{Rem}

%\begin{Rem}
%For general CE rational map, same result holds with $\alpha$ depending on the critical orbit we are following. In other words, one need to replace $\underline{\gamma}$ by the lower Lyapunov exponent of the corresponding critical value, while $\bar{\gamma}$ stays the same as one can take $\bar{\gamma}$ to be the maximal derivative of the map on the Riemann sphere.
%\end{Rem}

Although a Collet--Eckmann map has controlled recurrence rates in the sense of Lemma~\ref{initial-lemma}, nearby parameters may return very deep due to the expansion. As mentioned before, we will delete these parameters in our analysis. For this purpose we need the notion of partition elements.

\begin{Def}[Partition elements]\label{partitions}
Let $S>0$ be given. A square $A\subset\QQ$ is called a partition element at time $n$ if the following holds for all $k\leq n$:
\begin{equation}
\diam\xi_{k}(A)\leq
\begin{cases}
~\,~\dfrac{\dist(\xi_{k}(A), 0)}{(\log \dist(\xi_{k}(A), 0))^2} ~&\text{if}\,~\xi_{k}(A)\cap U\neq \emptyset,\\
\quad~\,~S ~&\text{if}\,~\xi_{k}(A)\cap U= \emptyset.
\end{cases}
\end{equation}
\end{Def}
For clarification, in the rest of the paper we will write $A_{n}$ for a partition element at time $n$. More specifically, we use $A_n(c)$ for a partition element at time $n$ that contains the parameter $c$. It is useful to have in mind that partitions are getting finer and finer as time proceeds; in other words, $A_{n+1}(c)\subset A_n(c)$ for all $n$.

The term $S$ is referred as \emph{large scale} in this paper. It is always a good situation if a partition element reaches the large scale. However, this is not clear for all partition elements. Our main effort is to show that almost all partition elements will reach the large scale. Clearly, the obstruction to reach the large scale is caused by the frequent passages near the critical point. %The above definition gives us a \emph{partition rule}: Partition is made at each step as long as 

For a partition element, we can also give a notion of returns. Let $A$ be a partition element at time $k-1$. Then $\xi_k(A)$ is called a \emph{return} to $U$ if $\xi_k(A)\cap U\neq\emptyset$ and a {\em return} to $U'$ if $\xi_k(A)\cap U'\neq\emptyset$. It also may happen that $\xi_k(A)\cap U'\neq\emptyset$ but $\xi_k(A)\cap U=\emptyset$. In this case we say that $\xi_k(A)$ is a \emph{pseudo-return}. Now bound periods can also be defined similarly if a return occurs for a partition element.

\begin{Def}\label{bound_period}
Suppose that $A$ is a partition element at time $n$. The bound period for $A$ is the indices $j>0$ such that for any $c_1, c_2\in A$ and for any $z\in\xi_n(A)$ the following holds:
\[
|f_{c_1}^{j}(z)-\xi_j(c_2)|\leq e^{-\beta j}|\xi_j(c_2)|.
\]
The length of this bound period, denoted by $p$, is the largest number such that the above relation holds. 
\end{Def}

After bound periods end, the next return will be called a \emph{free return}. At this point, there are two possibilities: essential and inessential returns. Let $A_{n-1}$ be a partition element. A return $\xi_n(A_{n-1})$ to $U$ or $U'$ is \emph{essential} if
\[
\diam(\xi_n(A_{n-1})) \geq \frac{1}{2} \dfrac{\dist(\xi_{n}(A_{n-1}), 0)}{(\log \dist(\xi_{n}(A_{n-1}), 0))^2}.
 \]
 Otherwise the return is said to be \emph{inessential}.

 Suppose that $\xi_n(A_{n-1})$ is an essential return to $U$. Then partition at time $n$ is made from $A_{n-1}$ as follows. If $A_{n-1}$ is a partition element according to Definition~\ref{partitions} then we do not partition further, i.e. we put $A_{n} = A_{n-1}$. If not, we partition $A_{n-1}$ into four squares of equal size. If all of them are partition elements, then we are done. If not, we continue with each square not being a partition element, and partition it into four new squares of equal size. Continuing in this way we obtain a partition of $A_{n-1}$ into sub squares $A_n^j$ satisfying 
\begin{equation}\label{essele}
\frac{1}{2} \dfrac{\dist(\xi_{n}(A_n^j), 0)}{(\log \dist(\xi_{n}(A_n^j), 0))^2} \leq \diam(\xi_n(A_n^j)) \leq \dfrac{\dist(\xi_{n}(A_n^j), 0)}{(\log \dist(\xi_{n}(A_n^j), 0))^2}.
 \end{equation}
Thus each $A_{n}^{j}$ is a partition element at time $n$.

Now we make another definition. After the partition is made as above, for each set $\xi_n(A_n^j)$ put 
$\ti{r} = -\log (\dist(\xi_n(A_n^j),0))$. Define $r = \lceil \ti{r} - \frac{1}{2} \rceil$. Then we have 
 \[
e^{-r-1/2} \leq \dist(\xi_n(A_n^j), 0) < e^{-r+1/2}.
\]
If this is the case we also write $\dist(\xi_n(A_n^j),0) \sim e^{-r}$. In general we write $A \sim B$ if there is a non-dynamical constant $C$ such that
\[
\frac{1}{C} A \leq B \leq CA.
 \]
 
 \subsection{Free periods}
 
 After the end of a bound period, we enter into a \emph{free period} before the next return happens. During this time window, the orbit is away from critical points which gives growth of derivatives similarly as the classical Ma{\~n}\'e theorem. The following is a Ma{\~n}\'e type result which says that a Collet--Eckmann map is expanding away from critical points. A similar lemma was proved in~\cite{ABC-1}. 
 
 In the following, we state and prove the result for rational maps satisfying the so-called \emph{exponential shrinking property}. A rational map $f$ is said to satisfy this property if there exist $\lambda_{\Exp}>1$ and $r>0$ such that for every $z\in\JJ(f)$, each $n>0$ and every component $W$ of $f^{-n}(D(z,r))$ one has $\diam W\leq \lambda_{\Exp}^{-n}$. This property is equivalent to the \emph{topological Collet--Eckmann condition} and is, in general, weaker than the Collet--Eckmann condition; see~\cite{PRS}. For the setting considered in this paper, i.e. maps with exactly one critical point in the Julia set, these conditions are all equivalent. Let $\crit(f)$ be the set of critical points of $f$.

\begin{Lem}\label{eafcp}
Let $f$ be rational satisfying the exponential shrinking property. Then there exist a neighbourhood $U'$ of $\crit(f)\cap\JJ(f)$, $C>0$ and $\gamma>0$ such that the following holds: Let $z\in\JJ(f)$ and if $f^{j}(z)\not\in U'$ for $0\leq j\leq n-1$, then
$$|Df^{n}(z)|\geq Ce^{\gamma n}.$$
\end{Lem}

\begin{proof} Put $z_n = f^n(z)$ and $z_0 = z$. 
Since $f$ satisfies the exponential shrinking property, per definition one can find $\lambda_{\Exp}>1$ and $r>0$ such that preimage components shrink exponentially. Take any $\delta'\in(0,r]$ and put
$$W'_k:=f^{-k}(D(z_n,\delta')),$$
the component which contains $z_{n-k}$. Thus one has
$$\diam W'_k\leq \lambda_{\Exp}^{-k}$$
for all $k\leq n-1.$ So there exists $\ell\in\mathbb{N}$ such that
$$\diam W'_k \leq \lambda_{\Exp}^{-k}<2\delta'$$
for all $k\geq \ell$. Since $f^{j}(z)\not\in U'$ for $j<n$ which means that $|f^{j}(z)-c|>\delta'$ for all $j<n$ and all $c\in\crit(f)\cap\JJ(f)$, we see that
$$W'_{k}\cap(\crit(f)\cap\JJ(f))=\emptyset$$
for all $k\geq \ell$. This implies that there exists $0<\delta'':=\delta''(\delta')<\delta'$ such that if $W''_k\subset W'_k$ is the component of $f^{-k}(D(z_n,\delta''))$ we have
$$f^{k}: W''_k\to D(z_n,\delta'')$$
is conformal and onto for all $k\leq n$. By the above choice of $\delta''$ and the Koebe one-quarter theorem,
$$\lambda_{\Exp}^{-n}>\diam W''_n\geq \frac{1}{2}\frac{\delta''}{|Df^n(z_0)|}.$$
This gives
$$|Df^n(z)|\geq \frac{\delta''}{2}\lambda_{\Exp}^{n}=:Ce^{\gamma n}.$$
Here $C=\delta''/2$ and $\gamma=\log\lambda_{\Exp}$. 

So the lemma holds with $U'=\bigcup_{c\in\crit(f)\cap\JJ(f)}D(c,\delta')$.
\end{proof}

The following lemma follows from the above and is proved in~\cite[Lemma 3.10]{ABC-1}. To state the result, let $\mathcal{N}_{\varepsilon_0}$ be a neighbourhood of the Julia set $\JJ(f_{c_0})$, where $\varepsilon_0$ is chosen in a way such that $U'\subset \mathcal{N}_{\varepsilon_0}$. That is,
\begin{equation}\label{nbjulia}
\mathcal{N}_{\varepsilon_0}=\left\{z\in\C:\, \dist(z, \JJ(f_{c_0}))<\varepsilon_0 \right\}.
\end{equation}

Now we return to Collet--Eckmann maps in the unicritical family. Recall that $U, U', \Delta$ and $\Delta'$ are defined in~\eqref{uu'}.
 
 \begin{Lem}\label{oel}
Let $f_{c_0}$ be a Collet--Eckmann map. Then there exist $\Delta'>0$ sufficiently large, $C>0$ and $\gamma_H>0$ such that the following holds. For any $\Delta>\Delta'$ there exist $\varepsilon>0$ sufficiently small and $C_U>0$ (dependent on $U$) such that for all $n\geq 1$ and for all $c \in \mathcal{Q}$, if $ z,f_c(z),\dots f_c^{n-1}(z) \in \mathcal{N}_{\varepsilon_0} \setminus U$, then
\begin{equation}\label{outside_no_hit}
\vert Df_c^n(z) \vert \geq C_U e^{\gamma_H n}.
\end{equation}
If, in addition, $f_c^n(z) \in U$, then
\begin{equation}\label{outside_hit}
\vert Df_c^n(z) \vert \geq C e^{\gamma_H n}.
\end{equation}
\end{Lem}

\subsection{Transversality} One of the ingredients used in the proof of the theorem is a transversality result at Collet--Eckmann parameters. For the unicritical family, this was proved by A.~Avila~\cite{Avila2002} and G.~Levin~\cite{Levin2002a}. (For general rational Collet--Eckmann maps, this was due to G.~Levin~\cite{Levin-transv}.) This result says that for a Collet--Eckmann unicritical polynomial $f_{c_0}$, one has 
\[
\lim_{n\to\infty}\frac{\xi'_{n}(c_0)}{Df^{n-1}_{c_0}(c_0)}=L \,~(\neq 0,\, \infty).
\]
A direct consequence of this is the following lemma proved in~\cite[Proposition 4.1]{MA7}.

\begin{Lem} \label{levin}
 Let $f_{c_0}$ be a unicritical polynomial satisfying the Collet--Eckmann condition with $\ga_0 = \uli{\ga}(c_0)$. Then for any $q\in (0,1)$ and any $\gamma \in (0,\gamma_0)$ there exists $N_{L} >0$ and $\varepsilon >0$ such that
 \begin{equation} 
 \left\vert \frac{\xi'_{n}(c)}{Df^{n-1}_{c}(c)}-L \right \vert \leq q \vert L \vert
 \end{equation}
 provided that $f_c$ satisfies the Collet--Eckmann condition up to time $n\geq N_{L}$ with exponent $\gamma$ for all $c\in \QQ$. 
 \end{Lem}
 
It follows from this lemma that, with some a priori specified $\gamma<\gamma_0$, all parameters from $\QQ$ will have exponential growth along the critical orbit up to time $N_L$ with this exponent. Moreover, the function $\xi_n$ has bounded distortion for all $n\leq N_L$.

\subsection{The weak parameter dependence property}

We state here the basic and crucial fact, that, given some lower exponent $\ga_Q \in (0,\uli{\ga})$, and $Q > 1$ (typically so that $\log Q$ is much smaller than both $\uli{\ga}$ and $\ga_Q$), there is $\vep > 0$ (the size of the square $\QQ$) such that, if the derivative $|Df_c^k(c)| \geq e^{\ga_Q k}$ for all $N_L \leq k \leq n$, then
\begin{multline} 
Q^{-(n-k)} |Df_c^{n-k}(\xi_k(c))| |\xi_k(a) - \xi_k(b)| \\ \leq |\xi_n(a) - \xi_n(b)| \\ \leq Q^{n-k} |Df_c^{n-k}(\xi_k(c))| |\xi_k(a) - \xi_k(b)|,\label{weak-dist}
\end{multline}
for all $a,b,c$ belonging to the same partition element $A_n \subset \QQ$. For a detailed proof see Lemma 3.2 in~\cite{ABC-1}.
The point is that this result in connection to the Main Distortion Lemma~\ref{md} gives a strong version of~\eqref{weak-dist} which we present after Lemma~\ref{md}. 

\section{Bound periods, main distortion and start-up}

This section is devoted to proving distortion results. This relies on the expansion during bound periods which in turn depends on estimating the growth during this period (Lemma~\ref{bound-exp}). The expansion during bound periods depends on the length (Lemma~\ref{bound-length}) and distortion estimates (Lemma~\ref{bound-dist}) of bound periods.

\subsection{Distortion, length and expansion during bound periods}
We first prove that we have good distortion during the bound period (see Definition~\ref{bound_period}) and that, given a return at time $n$, this forces the bound period to never succeed $n$ in length. With these in hand it will be clear that one has some definite expansion when the bound periods end.

\paragraph{Distortion during bound periods.}
The following lemma gives distortion estimate in the bound periods.

\begin{Lem}[Distortion during the bound period]\label{bound-dist}
Let $\vep' > 0$ be given. Then if $\de' = e^{-\De'}$ is sufficiently small the following holds. Let $n$ be the index of a return, $A$ a partition element at time $n$, and let $p$ be the associated bound period. Then, for all $c \in A$ and all $z \in \xi_n(A)$,
\[
\left\vert \frac{(f_c^j)^\prime(f_c(z))}{(f_c^j)^\prime(c)} - 1 \right\vert \leq \vep' \qquad (j=1,2,\dots,p).
\]
\end{Lem}
\begin{proof}
By the chain rule and a standard estimate we have
\[
 \left\vert \frac{(f_c^j)^\prime(f_c(z))}{(f_c^j)'(c)} - 1 \right\vert \leq \exp\left\{\sum_{\nu = 0}^{j-1} \frac{\vert f_c^{\nu}(f_c(z))^{d-1} - f_c^{\nu}(c)^{d-1}\vert}{\vert f_c^{\nu}(c) \vert^{d-1}} \right\} - 1.
\]
Using the identity
\[
z^{d-1} - w^{d-1} = (z-w)\sum_{k = 0}^{d-2} z^{k}w^{d-2-k},
\]
we find that
\[
\vert f_c^{\nu}(f_c(z))^{d-1} - f_c^{\nu}(c)^{d-1} \vert \leq \vert f_c^{\nu}(f_c(z)) - f_c^{\nu}(c) \vert \sum_{k=0}^{d-2}\vert f_c^{\nu}(f_c(z)) \vert^k\vert f_c^{\nu}(c)\vert^{d-2-k}.
\]
During the bound period the following holds,
\[
\frac{\vert f_c^{\nu}(f_c(z))\vert}{\vert f_c^{\nu}(c)\vert} \leq \frac{\vert f_c^{\nu}(f_c(z)) - f_c^{\nu}(c) \vert}{\vert f_c^{\nu}(c) \vert} + 1 \leq e^{-\beta j} + 1 < 2,
\]
and hence,
\[
\vert f_c^{\nu}(f_c(z))^{d-1} - f_c^{\nu}(c)^{d-1} \vert \leq \vert f_c^{\nu}(f_c(z)) - f_c^{\nu}(c) \vert \vert f_c^{\nu}(c) \vert^{d-2} (2^{d-1}-1).
\]
We conclude that
\[
\left\vert \frac{(f_c^j)^\prime(f_c(z))}{(f_c^j)'(c)} - 1 \right\vert \leq \exp\left\{ (2^{d-1}-1) \sum_{\nu=0}^{j-1} \frac{\vert f_c^{\nu}(f_c(z)) - f_c^{\nu}(c)\vert}{\vert f_c^{\nu}(c)\vert}\right\} - 1.
\]
Suppose that $\dist(0, \xi_n(A)) \sim e^{-r}$. Then for all $z \in \xi_n(A)$ we have $\vert f_c(z)-c \vert \sim e^{-dr}$. We divide the sum into two parts $[0,j-1] = [0,J]\cup[J+1,j-1]$, where $J=\lceil r/(\alpha + \oli{\ga}) \rceil$ and where $\al$ is the upper bound of $\al_j = \al_j(c)$ for $c \in A$, satisfying $\al < \oli{\ga} - \uli{\ga}$. Then the first sum becomes
\begin{align*}
\sum_{\nu = 0}^{J} &\frac{\vert f_c^{\nu}(f_c(z)) - f_c^{\nu}(c)\vert}{\vert f_c^{\nu}(c) \vert} \leq C \sum_{\nu = 1}^J \frac{\vert Df_c^{\nu}(c)| |f_c(z) - c |}{e^{-\al \nu} } \\ & \leq C \sum_{\nu=0}^J e^{\oli{\ga} \nu}|f_c(z)-c| e^{\al \nu} \leq C \sum_{\nu=1}^J e^{(\oli{\ga} + \al) \nu - dr} \leq C' e^{r-dr} = C' e^{-(d-1)\Delta'},
\end{align*}
which can be made arbitrarily small if $\De'$ is large enough. For the estimate of the second sum we use the binding condition directly:
\[
\sum_{\nu = J+1}^{j-1} \frac{\vert f_c^\nu(f_c(z)) - f_c^\nu(c) \vert}{\vert f_c^\nu(c) \vert} \leq \sum_{\nu = J+1}^{j-1} e^{-\beta \nu} \leq C e^{-\beta (J+1)} \leq C\exp\left\{-\frac{\beta}{\alpha + \oli{\ga}}\Delta' \right\}.
\]
Adding these two estimates, making $\delta'$ sufficiently small, completes the proof.
\end{proof}

\paragraph{Length of bound periods.}
The bounded distortion above ensures that the bound period will never exceed its return time, as shown in the following result. Recall that $N_L$ was defined in Section 2.3 and $\EE_n(\gamma)$ was defined in Definition~\ref{EEga}.

\begin{Lem}[Length of the bound period] \label{bound-length}
Let $n\geq N_L$ and $\xi_n(c)$ be a return where $c\in \EE_n(\ga)$ for some $\ga > 0$. Suppose that $p$ is the associated bound period for this return. Then $p < n$.
\end{Lem}
\begin{proof}
Assume, on the contrary, that $p \geq n$. This means that $2n$ belongs to the bound period of $\xi_n(c)$. By Definition~\ref{pwbound}, we have
\[
\vert \xi_{2n}(c) - \xi_n(c) \vert \leq e^{-\beta n}\vert \xi_n(c) \vert.
\]
With $z = \xi_n(c)$ in Lemma~\ref{bound-dist}, we have the following lower bound estimate for the left hand side of the above inequality:
\begin{align*}
\vert \xi_{2n}(c) - \xi_n(c) \vert &= \vert f_c^n(\xi_n(c)) - f_c^n(0) \vert \\
&=\vert f_{c}^{n-1}(\xi_{n+1}(c)) - f_{c}^{n-1}(c)\vert \\
&\geq C' \vert (f_c^{n-1})'(c) \vert \vert \xi_{n+1}(c) -c \vert \\
%&= \vert f_c^\prime(\xi_n(c)) \vert \vert (f_c^{n-1})^\prime(f_c(\xi_n(c))) \vert \vert \xi_n(c) \vert \\
&=C' \vert (f_c^{n-1})'(c) \vert \vert \xi_n(c) \vert^d,
\end{align*}
where $C'>0$ is some constant which is arbitrarily close to $1$ if $\varepsilon$ is sufficiently small. So we have that
\[
C'e^{\gamma_{n-1}(n-1)-\alpha_n n d}\leq e^{-\beta n} e^{-\alpha_n n},
\]
which means that
\[
C'e^{\gamma_{n-1}(n-1)-\alpha_n n (d-1)}\leq e^{-\beta n}.
\]
By~\eqref{alpha_gamma_relation}, one can see that
\[
\frac{1}{d}C' e^{\gamma_n n} \leq e^{-\beta n},
\]
which is impossible since $\gamma_n \geq \gamma > 0$ and $n\geq N_L$ and $N_L$ can be taken sufficiently large by making $\varepsilon$ small. This completes the proof.
\end{proof}
We will assume that $\ti{N} \leq N_L$ (in fact the difference should be very large). 

\paragraph{Expansion during bound periods.}
We have the following lemma and its corollary ensuring the derivative growth of the critical orbit during bound periods.
 
\begin{Lem} \label{bound-exp}
Let $\xi_{n}(c) \in U'$ and $|\xi_{n}(c)| \sim e^{-r}$ and let $p$ be the associated bound period. Then for any fixed $\eta > 0$ the following hold if $\delta' > 0$ is sufficiently small:
\begin{equation}\label{bdforp}
 \frac{(1-\eta)dr}{\ga_p + \al_{p+1} + \be} \leq p \leq \frac{(1+\eta)dr}{\ga_{p-1} + \al_p + \be},
\end{equation}
and
\begin{equation}\label{expbp}
\left|Df_{c}^{p+1}(\xi_{n}(c))\right| \geq C e^{(1-\eta) \frac{\ga_p - (d-1)(\al_{p+1}+ \be)}{\ga_p + \al_{p+1} + \be} r}.
\end{equation}
\end{Lem}

\begin{proof}
Since $p$ is the bound period for the return $\xi_n(c)$, we have, by Definition~\ref{pwbound}, that
\begin{equation}\label{ufexp}
\left|f_{c}^{p+1}(\xi_n(c))-\xi_{p+1}(c)\right|\geq e^{-\beta(p+1)}|\xi_{p+1}(c)|.
\end{equation}
Therefore, by Lemma~\ref{bound-dist} there exists some constant $C'>0$ such that
\[
|Df_{c}^{p}(c)||\xi_{n+1}(c)-c|\geq C'e^{-\beta(p+1)}|\xi_{p+1}(c)|,
\]
which gives, since $|Df_{c}^{p}(c)|=e^{\gamma_p p}$ and $|\xi_{p+1}(c)|=e^{-\alpha_{p+1}(p+1)}$,
\begin{equation}\label{drest}
e^{\gamma_p p} e^{-dr}\geq C'' e^{-(\alpha_{p+1}+\beta)(p+1)},
\end{equation}
where $C''=C'e^{-1/2}$. By taking $\delta>0$ sufficiently small, so that $r$ is large enough, (and thus $p$ sufficiently large) we have the first inequality of~\eqref{bdforp}. To see that the second inequality in~\eqref{bdforp} also holds, we use Definition~\ref{pwbound} and Lemma~\ref{bound-dist} again at the time $p$. This gives similar estimate as above but with reversed inequalities. We leave details for interested readers.

To obtain~\eqref{expbp}, first we note that, by~\eqref{drest},
\[
e^{- r} \geq C''^{1/d} e^{-\frac{1}{d}(\al_{p+1} + \be)(p+1)} e^{-\frac{1}{d}\gamma_p p},
\]
which implies
\begin{equation}\label{drest1}
e^{-(d-1)r} \geq C''^{(d-1)/d} e^{ - \frac{d-1}{d}(\al_{p+1} + \be) (p+1) -\frac{d-1}{d}\gamma_p p}.
\end{equation}
Moreover,
\begin{align}
 \left|Df_{c}^{p+1}(\xi_{n}(c))\right| &= |Df_{c}^{p}(\xi_{n+1}(c))||Df_{c}(\xi_n(c))| \nonumber \\
                    &\geq C'' |Df_{c}^{p}(c)|e^{-(d-1)r} \nonumber \\
                    &\geq C''^{\frac{2d-1}{d}} e^{\ga_p p - \frac{d-1}{d}(\al_{p+1} + \be) (p+1) -\frac{d-1}{d}\gamma_p p} \nonumber \\
                    &\geq C''^2 e^{-2\overline{\gamma}} e^{\frac{1}{d}(\ga_p  - (d-1)(\al_{p+1} + \be))p} . \label{erest} 
% &\geq e^{\frac{\gamma}{2d} p},
\end{align}
Combining~\eqref{erest} with the first inequality of~\eqref{bdforp}, we get
\[
|Df_c^{p+1}(\xi_n(c))| \geq C e^{(1-\eta) \frac{ \ga_p - (d-1)(\al_{p+1} + \be)}{ \ga_p + \al_{p+1} + \be } r},
\]
where $C>0$ is a (non-dynamical) constant.
\end{proof}

\comm{
%\begin{comment}
\begin{Rem}
Note that, since $c \in \EE_{n+1}(\ga)$, if we impose a condition on $\be$, like $\be \leq \ga/(4d)$ we have, by Remark~\ref{alpha_gamma_relation2},
\begin{equation} 
\ga_p- (d-1)(\alpha_{p+1} + \be) \geq \ga - \be(d-1) - \frac{C}{p} \geq \frac{3\ga}{4} + \frac{\ga}{4d} - \frac{C}{p} > \frac{3\ga}{4}, 
\end{equation}
 if $p$ is large enough (i.e. if $\de' = e^{-\De'}$ is small enough). Again, provided $p$ is large enough, we can get rid of the constants in front of (\ref{erest}), and get 
 \[
 \left|Df_{c}^{p+1}(\xi_{n}(c))\right| \geq e^{\frac{3 \gamma}{4d} p}.
  \]
\end{Rem}
}
%\end{comment}

\begin{Rem} \label{p-not-return}
By the definition of the bound period, one can see that $p$ cannot be a return time. This means, for example, that $\al_p \approx 0$. 
\end{Rem}

Notice that~\eqref{alpha_gamma_2} gives us
\[
\gamma_{k} \geq \gamma_{k+1} + (d-1)\alpha_{k+1} - \frac{C}{k}.
\]
If we assume that $c \in \EE_n(\gamma)$ then, since $p < n$ by Lemma~\ref{bound-length},
\[
\gamma_p \geq \gamma + (d-1)\alpha_{p+1} - \frac{C}{p+1}.
\]
Therefore, assuming $\beta < \gamma/(4d)$,
\begin{equation} \label{quarter-ga}
\gamma_p - (d-1)(\alpha_{p+1} + \be) \geq \gamma - \beta(d-1) - \frac{C}{p+1} \geq \frac{3\gamma}{4},
\end{equation}
provided $p$ is large enough (which is the case if $\delta'$ is small enough). Now we assume that $\ga \geq \uli{\gamma}$, and $\al_{p+1} < \oli{\gamma} - \uli{\gamma} \leq \oli{\gamma}$ (see Remark~\ref{gamma-ulioli}). 
Choosing $\eta$ sufficiently small in Lemma~\ref{bound-exp} we get
\[
(1-\eta)\frac{\gamma_p - (d-1)(\alpha_{p+1} + \beta)}{\gamma_p + \alpha_{p+1} + \beta} \geq (1-\eta)\frac{3 \gamma}{\left(8+\frac{1}{d}\right)\overline{\gamma}} \geq \frac{\underline{\gamma}}{3 \overline{\gamma}}.
\]
We can get rid of the constant in front of (\ref{expbp}), such that 
 \[
|Df_{c}^{p+1}(\xi_n(c)) | \geq e^{\frac{\uli{\ga}}{4 \oli{\ga}} r},
\]
given that $\de'$ is sufficiently small.

Now, put
\[
\ka = 1 - \frac{\uli{\ga}}{4 \oli{\ga}}.
\]
%\begin{align} \label{kappa}
%1 - (1-\eta)\frac{ \ga_p - (d-1)(\al_{p+1} + \be)}{ \ga_p + \al_{p+1} + \be } &= \frac{\eta(\ga_p + (d-1)(\al_{p+1}+ \be)) + %d(\alpha_{p+1} + \beta)}{\gamma_p + \alpha_{p+1} + \beta} \\
%&\leq \frac{d(\alpha_{p+1} + \beta) + \eta \oli{\ga}}{(d-1) \alpha_{p+1} + \ga + \al_{p+1} + \beta - \frac{C}{p}}.
%\end{align}
%If we now use the condition $\be < \ga/(2d)$ we get, 
%\[
%1 - (1-\eta)\frac{ \ga_p - (d-1)(\al_{p+1} + \be)}{ \ga_p + \al_{p+1} + \be } \leq \frac{d(\alpha_{p+1} + \beta) + \eta %\oli{\ga}}{d (\alpha_{p+1} + \be) + \frac{\ga}{d} - \frac{C}{p}} \leq \ka < 1,
%\]
%if $\eta$ is sufficiently small ($p$ is large enough depending on $\de'$).
We arrive at the following corollary.

\begin{Cor} \label{bound-kappa}
 Suppose that $\xi_n(c) \in U'$, $|\xi_{n}(c)| \sim e^{-r}$ and $c \in \EE_{n}(\ga)$ for some $\ga \geq \uli{\ga}$. Moreover, suppose that $\be < \ga /(4d)$. Then
 \[
|Df_c^{p+1}(\xi_n(c))| \geq e^{(1-\ka)r}. 
  \]
 \end{Cor}

We finish this section with the following two useful lemmas.

\begin{Lem} \label{bound-order}
 Suppose that $\xi_{n_1}(c)$ and $\xi_{n_2}(c)$ are two returns, $c \in \QQ$, and that $|\xi_{n_1}(c)| \sim e^{-r_1}$ and
 $|\xi_{n_2}(c)| \sim e^{-r_2}$, where $r_1 \leq r_2 - 1$. Then the corresponding bound periods $p_1$ and $p_2$ satisfy $p_1 \leq p_2$. 
\end{Lem}

\begin{proof}
 By bounded distortion, Lemma~\ref{bound-dist} (note that we can choose $\vep'$ as small as we like, therefore $C$ below is very close to $1$), we have, for $j \leq \min\{p_1, p_2\}$, 
 \[
  |\xi_{n_1+j}(c) - \xi_{j}(c)| \geq C e^{-2r_1} |Df_c^j(c)|\geq C e^{2} e^{-2r_2} |Df_c^j(c)| \geq |\xi_{n_2+j}(c) - \xi_{j}(c)|.
\]
So the bound period for the second return will be longer than (or at least as long as) the first one. 
 \end{proof}

 \begin{Lem} \label{adjacent-returns}
  Suppose that $\xi_{n_1}(A)$ and $\xi_{n_2}(A)$ are two consecutive returns, $n_2 > n_1$, where $A$ is a partition element at time $n_2$ and that $A \subset \EE_{n_1}(\ga)$ for some $\ga > 0$, and $\be < \ga /(4d)$. Then $|Df_c^{n_2-n_1}(\xi_{n_1}(c))| \geq e^{\gamma_2 (n_2 - n_1)},$ for some $\ga_2 \geq \ga /(3d)$ and for $c\in A$. 
  In particular, 
  \[
\diam(\xi_{n_2}(A)) \geq 2 \diam(\xi_{n_1}(A)).
\]
Moreover, $|Df_c^{n_2}(c)| \geq e^{(\gamma/3) n_2}$ if $\gamma \leq \gamma_H$. 
  \end{Lem}
  \begin{proof}
   We use Lemma~\ref{bound-exp} and Lemma~\ref{oel} and equations~\eqref{erest} and~\eqref{quarter-ga} to conclude, with $p_1$ the bound period after $n_1$, and $\ell_1$ the following free period,
   \[
|Df_c^{p_1+1}(\xi_{n_1}(c))| \geq e^{\frac{3 \gamma_{p_1}}{4d} p_1 }. 
    \]
    and therefore, since $\gamma_{p_1} \geq \ga$, 
    \[
|Df_c^{n_2-n_1}(\xi_{n_1}(c))| \geq e^{\frac{3 \ga}{4d} p_1 } C e^{\gamma_H \ell_1} \geq e^{\gamma_2 (n_2-n_1)},
\]
for some $\gamma_2 \geq \ga/(3d)$. Moreover, by the weak parameter dependence property, 
   \[
\diam(\xi_{n_1+p_1+\ell_1}(A)) \geq CQ^{-(p_1 + \ell_1)} e^{\ga_H \ell_1} e^{(1-\ka)r} \diam(\xi_{n_1}(A)) \geq 2 \diam(\xi_{n_1}(A)) ,
\]
since $r$ is large, and $Q > 1$ is comparatively small ($\log Q < \min\{\ga_H, \ga/(3d) \}$). Since $p_1< n_1$ by Lemma~\ref{bound-length}, using that $n_1$ is large, 
 \[
|Df_c^{n_2}(c)| \geq Ce^{\ga n_1 + \frac{3 \ga}{4d} p_1 + \gamma_H \ell_1} \geq e^{(\gamma/3) n_2}.
\]
   \end{proof}

%\section{Expansions and parameter dependence}

\subsection{Main Distortion Lemma}

Let us first state the main distortion lemma, see Lemma 3.14,~\cite{ABC-1}. It can be used in our situation as well, even if the map is not slowly recurrent. The constant $\ti{\ga}$ in the following lemma can be chosen to be arbitrarily small, but we may have to adjust $\QQ$, i.e. $\vep$ to it. We want to be able to use the lemma related to a small ``critical'' $\ga_C > 0$ defined later. 
%We will later define the constant $\gamma_C > 0$ as a kind of critical exponent, and we will use the lemma below for $\ti{\ga} \geq %3 \ga_C$. 

\begin{Lem}[Main distortion lemma]\label{md}
Let $\varepsilon'>0$ and $\ti{\gamma} > 0$. Then there exists $N$ large enough such that the following holds: If $A\subset\EE_{\nu}(\ti{\gamma})$ is a partition element and $\nu \geq N$ is a return time or does not belong to the bound period, and $\nu'$ is the next free return, then we have
\begin{equation} 
\left\vert \frac{Df_{c_1}^{n}(c_1)}{Df_{c_2}^{n}(c_2)} -1\right \vert \leq\varepsilon'
\end{equation}
for $c_1,c_2\in A$ and for $\nu\leq n\leq\nu'$, provided that $A$ is still a partition element at the time $n$.
\end{Lem}

For consistency we include the proof of this important lemma, although its proof is almost the same, verbatim, as the proof of Lemma 3.14 in~\cite{ABC-1} (alternatively see~\cite[Lemma 7.3]{MA7}). 
\begin{proof}
The exponent cannot drop too much from $\nu$ until $\nu'$ due to Lemma~\ref{adjacent-returns}. In fact, the exponent $\ga_n = \ga_n(c)$ at $n$, satisfies $\ga_n \geq \ti{\ga}/3$. Note that the weak parameter dependence property holds as soon as $c\in \EE_n(\ti{\ga}/3)$ (if we choose $\ga=\ga_Q >0$ small enough in Lemma~\ref{levin}), so we can use Lemmas 3.2 and 3.4 in~\cite{ABC-1}. The lemma then reduces to check whether the following sum can be made arbitrarily small:
\begin{equation}
\Upsilon:=\sum_{j=1}^{n-1}\frac{\vert \xi_{j}(a)-\xi_{j}(b) \vert}{|\xi_{j}(b)|}.
\end{equation}

We divide the sum into three parts as follows. Let $(\nu_k)$ be the free returns before time $n$, where $k\leq s$. In other words, $\nu=\nu_s$ and $\nu'=\nu_{s+1}$. Let also $p_k$ be the length of the associated bound period of the return $\nu_k$. The estimate of $\Upsilon$ is then divided into the following parts:
\begin{align*}
\Upsilon=\sum_{k=1}^{s}\sum_{j=\nu_{k-1}}^{\nu_{k-1}+p_{k-1}}\frac{\vert \xi_{j}(a)-\xi_{j}(b)\vert}{|\xi_{j,l}(b)|}&+\sum_{k=1}^{s}\sum_{j=\nu_{k-1}+p_{k-1}+1}^{\nu_{k}-1}\frac{\vert \xi_{j}(a)-\xi_{j}(b)\vert}{|\xi_{j}(b)|} \\
&+\sum_{j=\nu_s}^{n-1}\frac{\vert \xi_{j}(a)-\xi_{j}(b)\vert }{|\xi_{j}(b)|} \\
&=:\Upsilon_B + \Upsilon_F + \Upsilon_T.
\end{align*}
Here $\Upsilon_B$ denotes the contribution from bound periods, while $\Upsilon_F$ the contribution from free periods, and $\Upsilon_T$ the contribution from the last return $\nu_s$ up until time $n$.

For the bound periods, we use the definition of bound period, distortion, to estimate each sum in the sums in $\Upsilon_B$. First we note that, by the weak parameter dependence property, if $|\xi_{\nu_{k-1}}(a) - \xi_{\nu_{k-1}}(b)| \sim e^{-r_{k-1}}/r_{k-1}^2$, 
\begin{align}
 |\xi_{\nu_{k-1}+1}(a) - \xi_{\nu_{k-1}+1}(b)| &\sim e^{-r_{k-1}(d-1)} |\xi_{\nu_{k-1}}(a) - \xi_{\nu_{k-1}}(b)| \\
 &\sim \frac{e^{-r_{k-1}d}}{r_{k-1}^2} \sim \frac{|\xi_{\nu_{k-1}+1}(b) - \xi_1(b)|}{r_{k-1}^2} .
\end{align}
Therefore, 
\begin{align}
\Upsilon_B^k = \sum_{j=\nu_{k-1}}^{\nu_{k-1}+p_{k-1}} &\frac{\vert \xi_{j}(a)-\xi_{j}(b)\vert}{|\xi_{j}(b)|} \lesssim \frac{1}{r_{k-1}^2} + \sum_{j=\nu_{k-1}+1}^{\nu_{k-1}+p_{k-1}} \frac{\vert \xi_{j}(a)-\xi_{j}(b)\vert}{|\xi_{j}(b)|} \\
                                                 &\lesssim \frac{1}{r_{k-1}^2} + \sum_{j=\nu_{k-1}+1}^{\nu_{k-1}+p_{k-1}} \frac{|Df^{j-\nu_{k-1}-1}(\xi_{\nu_{k-1} +1}(b))| \vert \xi_{\nu_{k-1}+1}(a)-\xi_{\nu_{k-1}+1}(b) \vert}{|\xi_j(b)|} \\
 &\lesssim \frac{1}{r_{k-1}^2} + \sum_{j=\nu_{k-1}+1}^{\nu_{k-1}+p_{k-1}} \frac{|Df^{j-\nu_{k-1}-1}(\xi_{\nu_{k-1} +1}(b))| e^{-r_{k-1}d}}{r_{k-1}^2|\xi_j(b)|} 
 \\
 &\lesssim \frac{1}{r_{k-1}^2} + \sum_{j=\nu_{k-1}+1}^{\nu_{k-1}+p_{k-1}} \frac{|Df^{j-\nu_{k-1}-1}(\xi_{\nu_{k-1} +1}(b))| |\xi_{\nu_{k-1}+1}(b) - \xi_{1}(b)| }{r_{k-1}^2|\xi_j(b)|} \\ 
&\lesssim \frac{1}{r_{k-1}^2} +
\sum_{j=\nu_{k-1}+1}^{\nu_{k-1}+p_{k-1}} \frac{|\xi_{j}(b) - \xi_{j-\nu_{k-1}}(b)|}{|\xi_j(b)|} \\
&\lesssim \frac{1}{r_{k-1}^2} +
                                                   \sum_{j=\nu_{k-1}+1}^{\nu_{k-1}+p_{k-1}} \frac{e^{-\be (j-\nu_{k-1})}}{r_{k-1}^2} \leq \frac{C}{r_{k-1}^2}.
\end{align}

Given $r\geq\Delta$, let $K(r)$ be the set of indices $k$ such that $|\xi_{\nu_k}(A)| \sim e^{-r}$, and let $\hat{k}(r)$ be the largest index contained in $K(r)$. Then it follows from Lemma~\ref{adjacent-returns} that
\begin{equation}
\Upsilon_B=\sum_{k=1}^{s-1}\Upsilon_{B}^{k}\leq\sum_{r\geq\Delta}\sum_{k\in K(r)}\Upsilon_{B}^{k}\lesssim \sum_{r\geq\Delta}\Upsilon_{B}^{\hat{k}(r)} \lesssim \sum_{r\geq\Delta} \frac{1}{r^2}\lesssim \frac{1}{\Delta},
\end{equation}

To estimate the contribution from the corresponding free period, we use the weak parameter dependence and Lemma~\ref{oel} to obtain
\begin{align*}
\vert \xi_{\nu_{k}}(a)-\xi_{\nu_{k}}(b)\vert &\geq \frac{1}{Q^{\nu_{k}-j}} \vert Df_{a}^{\nu_{k}-j}(\xi_{j}(a))\vert \xi_{j}(a)-\xi_{j}(b)\vert \\
&\gtrsim \left(\frac{e^{\gamma_H}}{Q}\right)^{\nu_{k}-j} \vert \xi_{j}(a)-\xi_{j}(b) \vert
\end{align*}
for $\nu_{k-1}+p_{k-1}+1\leq j\leq \nu_{k}-1$. Since $Q$ is small compared to $e^{\ga_H}$ we get the each sum $\Upsilon_F^k$ in $\Upsilon_F$ can be estimated by 
\[
\Upsilon_{F}^{k}\lesssim \frac{\vert \xi_{\nu_{k}}(a) - \xi_{\nu_{k}}(b)\vert}{e^{-r}} \sum_{j=\nu_{k-1}+p_{k-1}+1}^{\nu_{k}-1}\left(\frac{Q}{e^{\gamma_H}}\right)^{\nu_{k}-j}\lesssim \frac{\vert \xi_{\nu_{k}}(a) - \xi_{\nu_{k}}(b)\vert}{e^{-r}}.
\]
Using the same argument as in the estimate of the contribution from the bound periods, we find that
\[
\Upsilon_F=\sum_{k=1}^{s-1}\Upsilon_{F}^{k}\leq\sum_{r\geq\Delta}\sum_{k\in K(r)}\Upsilon_{F}^{k}\lesssim \sum_{r\geq\Delta}\Upsilon_{F}^{\hat{k}(r)} \lesssim \sum_{r\geq\Delta} \frac{1}{r^2}\lesssim \frac{1}{\Delta}.
\]

We now consider the tail $\Upsilon_T$. It is easy to see that we may assume that $\nu_{s}+p_s+1\leq n< \nu_{s+1}$, otherwise $n$ belongs to a bound period or is simply a free return which we already have considered. For this purpose, we consider pseudo-returns, and we let $\nu_{s}+p_s+1\leq q_1\leq\cdots\leq q_t\leq n$ be the indices of these returns. By definition, $\xi_{q_k}(a)\cap U'\neq\emptyset$ and $\xi_{q_k}(a)\cap U=\emptyset$. For pseudo-returns, bound periods and free periods are defined in a similar way. As in the previous estimates, the contribution to the distortion between any two pseudo-returns of index $q_k$ and $q_{k+1}$ is a constant times
\[
\frac{\vert \xi_{q_k}(a) - \xi_{q_k}(b) \vert }{e^{-r_k}} + \frac{\vert \xi_{q_{k+1}}(a) - \xi_{q_{k+1}}(b) \vert}{e^{-r_{k+1}}},
\]
where $r_k$ and $r_{k+1}$ such that $|\xi_{q_k}(b)| \sim e^{-r_k}$ and $|\xi_{q_{k+1}}(b)| \sim e^{-r_{k+1}}$. The difference here is that, at a pseudo-return, the only thing we know about the length of our interval is that $\vert \xi_{q_k}(a) - \xi_{q_k}(b) \vert \leq S$, where $S = \varepsilon_1 \delta$ is the large scale. With similar methods and notation used for estimating the bound and free contributions, we have
\begin{align*}
\Upsilon_T &= \left(\sum_{j = \nu_s}^{q_1} + \sum_{k = 1}^{t-1} \sum_{j = q_k}^{q_{k+1}-1} + \sum_{j = q_t}^{n-1}\right) \frac{ \vert \xi_{j}(a) - \xi_{j}(b)\vert}{| \xi_{j}(b)|} \\
&\lesssim \frac{1}{r_s^2} + \sum_{k = 1}^t \frac{\vert \xi_{q_k}(a) - \xi_{q_k}(b)\vert}{| \xi_{q_k}(b)|} + \sum_{j = q_t}^{n-1} \frac{ \vert \xi_{j}(a) - \xi_{j}(b)\vert}{| \xi_{j}(b)|} \\
&\lesssim \frac{1}{\Delta^2} + \sum_{r = \Delta'}^\Delta \frac{\vert \xi_{q_{\hat{k}(r)}}(a) - \xi_{q_{\hat{k}(r)}}(b) \vert}{| \xi_{q_{\hat{k}(r)}}(b)|} + \frac{S}{\delta'} \\
&\lesssim \frac{1}{\Delta^2} + \varepsilon_1 \sum_{r = \Delta'}^\Delta e^{r - \Delta} + \varepsilon_1 \lesssim \frac{1}{\Delta^2} + \varepsilon_1,
\end{align*}
where we in the sum from $q_t$ to $n-1$ used Lemma~\ref{oel} (the first inequality, now with respect to $U'$) and that $|\xi_{j}(b)| > \delta' > \delta$ during this time.

Combining all these estimate above we arrive at 
\[
\Upsilon = \sum_{j=1}^{n-1} \frac{\vert \xi_{j}(a) - \xi_{j}(b) \vert}{| \xi_{j}(b)|} \lesssim \frac{1}{\Delta} + \varepsilon_1,
\]
and if $\delta$ and $\varepsilon_1$ are small enough, we reach the desired conclusion of strong distortion.

%We estimate all of these three parts using distortion lemmas, weak parameter dependence precisely as in the proof of Lemma 3.14 in %\cite{ABC-1}. So we refer the reader to this paper for the details. 

\end{proof}

From this lemma we can now deduce the strong version of~\eqref{weak-dist}, namely the following. For any $Q > 1$, there exists $\vep > 0$, $N_L > 0$, and $\de' > 0$ such that, for $N_L \leq k \leq n$ we have  
\begin{multline} 
Q^{-1} |Df_c^{n-k}(\xi_k(c))| |\xi_k(a) - \xi_k(b)| \\ \leq |\xi_n(a) - \xi_n(b)| \\ \leq Q |Df_c^{n-k}(\xi_k(c))| |\xi_k(a) - \xi_k(b)| \label{strong-dist}
\end{multline}
for all $a,b,c \in A_n$ belonging to the same partition element.

\subsection{Start-up}

Next we prove the following lemma that provides us with a suitable initial parameter square. (Compare with Lemma~3.15 in~\cite{ABC-1}.)

\begin{Lem}\label{startup}
Let $f_{c_0}$ be a Collet--Eckmann unicritical map. Let $N_L$ be as in Lemma~\ref{levin}, let $\varepsilon', \varepsilon_1>0$ be sufficiently small. Then there is a neighbourhood $U$ of the critical point such that, with $S = \varepsilon_1 \delta$ and for each sufficiently small $\varepsilon>0$ there is $N, \ti{N}'$ such that $\ti{N}' \leq N_L \leq N$ and such that for all $c\in\QQ$ we have the following:
\begin{enumerate}[(i)]
\item[$(i)$] For some $\gamma \geq \uli{\gamma}(c_0)(1-\varepsilon')$, one has for all $c \in \QQ$
\[
\left \vert Df_{c}^{k}(c)\right \vert \geq e^{\gamma k} \quad (\ti{N}' \leq k \leq N-1).
\]
%and
%\[
%\left\vert Df^{N-1}_c(c) \right\vert \geq e^{\gamma (N-1)}.
%\]

\item[$(ii)$] $\mathcal{Q}$ is a partition element at time $N-1$, and $\xi_N(Q)$ either is an essential return or reaches the large scale $S$.
%\begin{equation}
%\diam\xi_{N}(\QQ)\geq
%\begin{cases}
%~\,~\dfrac{\dist(\xi_{N}(\QQ),0)}{(\log \dist(\xi_{N}(\QQ), 0))^2},~&\text{if}~\xi_{k}(\QQ)\cap U\neq \emptyset,\\
%\quad~\,~S, ~&\text{if}~\xi_{k}(\QQ)\cap U= \emptyset;
%\end{cases}
%\end{equation}

\item[$(iii)$] For all $c_1,c_2\in\QQ$ one has
\[
\left\vert \frac{Df_{c_1}^{k}(c_1)}{Df_{c_2}^{k}(c_2)}-1 \right \vert\leq\varepsilon' \quad (k = 0,1,\dots N).
\]

\item[$(iv)$] There exists a constant $K = K(\delta,\varepsilon_1) > 1$ such that for all $c \in \mathcal{Q}$ and $0 \leq k \leq N-1$
\[
K^{-1} \vert \xi_k(c) \vert \leq \vert \xi_k(c_0) \vert \leq K \vert \xi_k(c) \vert.
\]
\end{enumerate}
\end{Lem}
We will use this lemma where $\vep' \leq 1/10$, so that $\ti{N}'$ can be set equal to $\ti{N}$ in Definition~\ref{EEga}.

\begin{proof}
For any complex $c_1,c_2,z,w$ we have
\[
\left\vert \frac{Df_{c_1}(z)}{Df_{c_2}(w)} - 1\right\vert \leq \left\vert \frac{z-w}{w^{d-1}}\sum_{k=0}^{d-2} z^k w^{d-2-k} \right\vert \leq \left\vert \frac{z-w}{w} \right\vert \sum_{k=0}^{d-2} \left(\left\vert \frac{z-w}{w} \right\vert + 1\right)^k.
\]
Assuming that $z$ and $w$ belong to a partition element, we have that
\[
\left\vert \frac{z-w}{w} \right\vert \leq 
\begin{cases}\, \Delta^{-2}\ &\mbox{if}\ z \in U\ \mbox{or}\ w \in U,\\
\,\varepsilon_1\ &\mbox{if}\ z,w \in \mathbb{C}\smallsetminus U.\end{cases}
\]
With $\eta = \min(\Delta^{-2},\varepsilon_1)$ we therefore have
\[
\left\vert \frac{Df_{c_1}(z)}{Df_{c_2}(w)} - 1\right\vert \leq (1+\eta)^{d-1} - 1.
\]
In particular,
\[
\left\vert \frac{Df_{c_1}(z)}{Df_{c_2}(w)} \right\vert \leq C,
\]
with $C > 1$ arbitrarily close to $1$ if $\delta$ and $\varepsilon_1$ are sufficiently small. Assuming that $\QQ$ is a partition element at time $n$, and that $c \in \QQ$, we get from the above distortion that
\[
\vert Df^k_c(c) \vert \geq C^{-k} \vert Df_{c_0}(c_0) \vert \geq C^{-k}C_0 e^{\gamma_0 k} \geq C_0 e^{\gamma k}.
\]
This tells us that as long as $\QQ$ is a partition element $f_c$, with $c \in \QQ$, will satisfy a Collet--Eckmann condition. We may thus apply Lemma~\ref{levin} to conclude that our interval will expand:
\[
\diam\left( \xi_k(\QQ)\right) \sim \vert \xi_k^\prime(c')\vert \diam\left( \QQ \right) \sim e^{\gamma k} \diam\left( \QQ \right).
\]
In particular we can make the perturbation small enough to find the time $N$ such that (i) and (ii) are valid.

By the choice of $N_L$, we can choose $\varepsilon>0$ sufficiently small such that $\QQ\subset\EE_{N_L}(\gamma)$ for some $\gamma$ arbitrarily close to $\gamma_0$.

By the definition of a partition element we can use Lemma~\ref{md} repeatedly starting from the time $N_L$ up to $N-1$ to get the distortion claimed in (iii).

To prove (iv) suppose first that $\xi_k(\QQ)$ is a return to $U$, and let $r$ be the integer satisfying
\[
e^{-r} \leq \dist(\xi_k(\mathcal{Q}),0) \leq e^{-r+1}.
\]
Then $\diam \left(\xi_k(\mathcal{Q})\right) \leq e^{-r+1}/r^2$, hence for every $c \in \mathcal{Q}$ we have that
\begin{align*}
\vert \xi_k(c) \vert &\geq \vert \xi_k(c_0) \vert - \diam \xi_k(\mathcal{Q}) \\
&\geq \left( 1 - \frac{e^{-r+1}}{\vert \xi_k(c_0) \vert r^2}\right)\vert \xi_k(c_0) \vert.
\end{align*}
Since $\vert \xi_k(c_0) \vert \geq \dist(\xi_k(\mathcal{Q}),0) \geq e^{-r}$ and $r \geq \Delta$, we conclude that
\[
\vert \xi_k(c) \vert \geq \left(1 - \frac{e}{\Delta^2} \right) \vert \xi_k(c_0) \vert.
\]
On the other hand, if we are outside of $U$, then
\[
\vert \xi_k(c) \vert \geq \left(1-\frac{S}{\delta}\right) \vert \xi_k(c_0) \vert = (1-\varepsilon_1)\vert \xi_k(c_0)\vert.
\]
This proves one of the inequalities, and the other one follows from exchanging $c_0$ and $c$ in the above.
\end{proof}

\section{Promotion}

In this part we introduce the key idea of \emph{promotion} which means that most of the parameters in $\QQ$ will reach a slow recurrence condition after several essential returns.

By Lemma~\ref{startup}, all critical orbits for parameters in $\QQ$ will follow the starting critical orbit up to time $N$ with bounded distortion and exponential growth of derivatives. If $\QQ$ has grown to the large scale under $\xi_N$, then $\xi_N(\QQ)$ covers mostly the Fatou set of $f_0$, and by a compactness argument, we can conclude that most parameters in $\QQ$ will land in the Fatou set of the corresponding map under $\xi_n$; See the next section for more details. Hence a very large portion of parameters in $\QQ$ are hyperbolic. Therefore, we may assume, without loss of generality, that $\xi_N(\QQ)$ has not reached the large scale yet and thus is an essential return.
%By decreasing $\varepsilon$ if necessary, we can also assume that $\QQ$ is a partition element at time $N$. In other words, $\xi_{N}(\QQ)$ satifies %\eqref{essele}.

At this time we may have to delete parameters that returns too deep at time $N$ in comparison to the starting map. Let $C_1 > 1$ be given and consider those $c \in \QQ$ such that 
\begin{equation}\label{appr2-ess}
\frac{\alpha_N(c)}{\alpha_{N}(c_0)}\leq C_1.
\end{equation}
This means that we delete parameters of measure
\[
 m\left(\{c \in \QQ : c \text{ does not satisfy~\eqref{appr2-ess} }\} \right) \lesssim m(\QQ) \frac{e^{-2C_1 \al_N(c_0) N}}{e^{-2 \al_N(c_0) N}}=\colon\delta_0(N) m(\QQ).
\]
Here, since $\alpha_N(c_0)$ is bounded, $\delta_0 = \delta_0(N)$ can be made arbitrarily small given $N$ large enough. We make the deletion so that the smallest collection of whole partition elements are deleted (this means that we may delete a little more than necessary). The remaining parameters are partitioned into partition elements as above. Each of these partition elements, containing $c$, is denoted by $A_N(c)$. Sometimes we just write $A_N$. However, note that $A_{N-1} = \QQ$ (there is no partition before time $N$). 

After this time $N$ we lose this kind of dependence on the starting orbit, hence we need to control the rate of recurrence in order to not lose too much growth of the derivative.

Let $N=\nu_0<\nu_1<\cdots$ be the free returns after time $N$. Suppose that $A = A_{\nu_j}$ is a partition element at time $\nu_j$. We will inductively introduce two sequences $\{\wt{\alpha}_{\nu_j}(A)\}$ and $\{\underline{\gamma}_{\nu_j}(A)\}$ which capture two key features for later analysis. Firstly, parameters $c \in A$ for which $\xi_{\nu_j}(c)$ are at distance $e^{-\wt{\alpha}_{\nu_j}(A)\nu_j}$ away from the critical point will be kept (and thus has controlled approach rate to the critical point), and secondly, we do not lose a lot of the derivative growth between two free returns.

We start with the time $\nu_0$. By Lemma~\ref{startup} (iv), $\xi_k(c)$ stays close to $\xi_k(c_0)$ and thus $\alpha_k(c)$ is comparable to $\alpha_k(c_0)$ for all $k\leq N$ and all $c\in A_{\nu_0}$.

\subsection*{Definition of $\wt{\alpha}_{\nu_j}(A)$.}
Put
$$\wt{\alpha}_{\nu_0}(A_{\nu_0})=\sup_{c\in A_{\nu_0}}\alpha_{\nu_0}(c).$$
Let $p_0$ and $\ell_0$ be the length of bound and free periods for the return time $\nu_0$, respectively. The following lemma is a consequence of Lemma~\ref{bound-exp}.

\begin{Lem} \label{bexp1}
Let $|\xi_{\nu_0}(c)| \sim e^{-r}$, for some $c \in A_{\nu_0}$, where $A_{\nu_0} \subset \QQ$ is a partition element satisfying (\ref{appr2-ess}). Then $|Df_{c}^{p_0+1}(\xi_{\nu_0}(c))| \geq e^{(1-\ka) r}$.
\end{Lem}

\begin{proof}
We can make $C_1$ as close to $1$ as we want so that $c \in A_{\nu_0}$ implies $c \in \EE_n( \uli{\ga})$. The lemma now follows from Corollary~\ref{bound-kappa}. 
 \end{proof}

%Put
%$$\kappa=1-\frac{ \uli{\ga}}{2d \oli{\ga}}.$$
Note that by the definition of $\ka$, Lemma~\ref{bexp1} says that $\diam \left(\xi_{\nu_0+p_0}(A_{\nu_0})\right)\geq e^{-\kappa r}$, where $\diam \left(\xi_{\nu_0}(A_{\nu_0})\right)\sim e^{-r}$. Now we fix some $\ka'$ such that $\ka < \ka' < 1$. We shall now define $\wt{\al}_{\nu_j}(A)$, where $A$ is a partition element at the free return time $\nu_j$.

Assume that $\dist(0, \xi_{\nu_j}(A)) \sim e^{-r_j} \sim e^{-\al_{\nu_j} \nu_j} $, and $A = A_{\nu_j}$ is a partition element. The lengths of the bound and free periods for this return are $p_j$ and $\ell_j$, respectively. Define
\begin{equation}\label{ta}
\wt{\al}_{\nu_{j+1}}(A) = \ka' \frac{\nu_j}{\nu_{j+1}} \sup\limits_{c \in A} {\al_{\nu_j}(c)} \quad\left( \leq \ka' \sup\limits_{c \in A} \al_{\nu_j}(c)\right).
\end{equation}
We do this at each return (including inessential returns). 

Note that since $\al_{\nu_j}(c)$ is almost constant in $A$, we have, for some constant $\ti{\ka}$ arbitrarily close to $1$ with $1>\ti{\ka}>\ka'$, that for any $c,\, c'\in A$,
$$\ti{\ka} \alpha_{\nu_j}(c')\leq \alpha_{\nu_j}(c)\leq \frac{1}{\ti{\ka}}\alpha_{\nu_j}(c').$$
Indeed, using the mean value theorem we find that for any $c,c' \in A$
\[
\left\vert \frac{\alpha_n(c)}{\alpha_n(c')} - 1 \right\vert \leq \frac{\left(\sup_{c \in A_{\nu_j}} e^{\alpha_n(c) n}\right) \vert \xi_n(c)- \xi_n(c') \vert}{\alpha_n(c') n } \sim \frac{1}{r_j^3}
\]
This, together with~\eqref{ta}, gives
\begin{equation}\label{tap}
\ti{\ka}\ka' \frac{\nu_j}{\nu_{j+1}} \al_{\nu_j}(c) \leq \wt{\al}_{\nu_{j+1}}(A) \leq \frac{\ka'}{\ti{\ka}}\frac{\nu_j}{\nu_{j+1}} \al_{\nu_j}(c)
\end{equation}
for all $c \in A$. 

At the next return we are going to delete from $A$ (which is a partition element at time $\nu_j$) parameters not satisfying
\begin{equation} \label{tilde-alpha}
|\xi_{\nu_{j+1}}(c)| \geq e^{-\wt{\al}_{\nu_{j+1}}(A) \nu_{j+1}}.
\end{equation}
In other words, if $c$ is deleted, then
\[
\al_{\nu_{j+1}}(c) > \ti{\al}_{\nu_{j+1}}(A).
 \]

We do this exclusion at each free return (actually it is going to be each essential return), see Lemma~\ref{freein} below. For a partition element $A$ at the return $\nu_j$, let us also define
\[
 \al_{\nu_j}(A) = \sup\limits_{c \in A} \al_{\nu_j}(c).
\]
So we have a procedure; given $A = A_{\nu_j}$ for some essential return $\xi_{\nu_j}(A)$ with $\dist(\xi_{\nu_j}(A),0) \sim e^{-\al_{\nu_j}(A) \nu_j}$, this induces an $\wt{\al}_{\nu_{j+1}}(A)$ at the next return, i.e. we have a mapping
\[
\al_{\nu_j}(A_{\nu_j}) \mapsto \wt{\al}_{\nu_{j+1}}(A_{\nu_j}),
\]
so that
\[
\al_{\nu_{j+1}}(A_{\nu_{j+1}}) \leq \wt{\al}_{\nu_{j+1}}(A_{\nu_j})
\]
for all partition elements $A_{\nu_{j+1}} \subset A_{\nu_j}$ that did not get deleted according to the rule (\ref{tilde-alpha}). 

%We first note that~\eqref{tilde-alpha} ensures some definite growth.
%if $A$ is a partition element at time $\nu_j$, then $\xi_{\nu_j+p_j}(A)$ will grow (at least) to the size of $e^{-\ka r_j}$.

\begin{Lem}\label{gogo}
Let $c\in A$ satisfy~\eqref{tilde-alpha}. Then $|Df_{c}^{p_{j+1}+1}(\xi_{\nu_{j+1}}(c))| \geq e^{\frac{ \uli{\ga}}{2d \oli{\ga}} r_j}$.
\end{Lem}

\begin{proof}
For $c$ satisfying~\eqref{tilde-alpha}, by~\eqref{tap} we have
\begin{equation}
|\xi_{\nu_{j+1}}(c)| \geq e^{-\wt{\al}_{\nu_{j+1}}(A) \nu_{j+1}}\geq e^{-(\ka'/\ti{\ka})\alpha_{\nu_j}(c)\nu_j}>e^{-\alpha_{\nu_j}(c)\nu_j}=|\xi_{\nu_j}(c)|,
\end{equation}
where we have used the fact that $\ka'<\ti{\ka}$ in the third inequality. Inductively, we see that $|\xi_{\nu_{j+1}}(c)|>|\xi_{\nu_0}(c)|$. By Lemma~\ref{bound-order}, the length of the bound period for $\nu_{j+1}$ satisfies $p_{j+1}\leq p_0$, where $p_0$ is the length of the bound period for the return at the start-up time $\nu_0$.

Then the lemma follows in the same way as in the proof of Lemma~\ref{bexp1}.
\end{proof}

It turns out that parameter exclusions happen only at essential returns due to the following simple observation.

\begin{Lem}\label{freein}
Let $A$ be a partition element at time $\nu_j$. If the free return $\xi_{\nu_{j+1}}(A)$ is inessential, then~\eqref{tilde-alpha} is satisfied for all $c\in A$.
\end{Lem}

\begin{proof}
To prove this happens, $\alpha_{\nu_{j+1}}(c)\leq \wt{\al}_{\nu_{j+1}}(A)$ for all $c\in A$. Assume that $|\xi_{\nu_{j+1}}(c)|\sim e^{-r_{j+1}}=e^{-\alpha_{\nu_{j+1}}(c)\nu_{j+1}}$. By the definition of inessential returns (cf. Section 2), we have
$$\diam\xi_{\nu_{j+1}}(A)\leq \frac{1}{2}\frac{e^{-r_{j+1}}}{r_{j+1}^{2}}.$$
By the above Lemma~\ref{gogo}, we also have
$$\diam\xi_{\nu_{j+1}}(A)\geq e^{-\ka r_{j}}.$$
Combining the above we obtain
$$\ka\alpha_{\nu_j}(c)\nu_j=\ka r_j \geq r_{j+1}=\alpha_{\nu_{j+1}}(c)\nu_{j+1},$$
which gives
$$\alpha_{\nu_{j+1}}(c)\leq \ka\alpha_{\nu_j}(c)\frac{\nu_j}{\nu_{j+1}}\leq\frac{\ka}{\ka'}\wt{\al}_{\nu_{j+1}}(A)<\wt{\al}_{\nu_{j+1}}(A)$$
since $\ka<\ka'$.
\end{proof}

Recall that $A = A_{\nu_j}$ is a partition element at time $\nu_j$ so we set $\wt{\al}_k(A) = \wt{\al}_{\nu_{j+1}}(A)$ for $\nu_{j} < k \leq \nu_{j+1}$, so that $\wt{\al}_k(A)$ is defined also between returns. 

\subsection*{Definition of $\uli{\gamma}_{\nu_j}$.}
% We start also with the first essential return time $\nu_0=N$.
Recall inequality~\eqref{alpha_gamma_1}: For $n$ large enough we have
\[
\gamma_n \geq \gamma_{n-1} - (d-1)\alpha_{n} - \frac{C}{n}.
\]
Let $A_n$ be a partition element at time $n$ and put
\[
 \gamma_{n}(A_n) = \inf\limits_{c \in A_n} \gamma_n(c) \quad \text{ and } \quad
 \gamma_{n-1}(A_n) = \inf\limits_{c \in A_n} \gamma_{n-1}(c) 
\]
We define
\[
\uli{\gamma}_{\nu_0}(A_{\nu_0})=\gamma_{\nu_0-1}(A_{\nu_0}) - (d-1) \al_{\nu_0}(A_{\nu_0}) - \frac{C}{\nu_0},
\]
so that $\uli{\gamma}_{\nu_0}(A_{\nu_0}(c_0)) = \uli{\gamma}_{N}$ is at least $\ga$ from Lemma~\ref{startup} (i) (for some suitable $\vep' \leq 1/10$). For a partition element $A$ at time $\nu_j$ we inductively define
\begin{equation} \label{gamma-bar}
  \uli{\ga}_{\nu_{j+1}}(A) = \left( \ga_{\nu_j}(A) + \ga_H \frac{\ell_j}{\nu_j}\right) \frac{\nu_j}{\nu_{j+1} - 1},
%
%  \left\{ \begin{array}{ll} \uli{\ga}_{\nu_j}(A) \dfrac{\nu_j}{\nu_{j+1}},& \text{ if } \ell_j \leq \nu_j; \\
 %
%  ~\\
 % 
        %        \max\left\{ \uli{\ga}_{\nu_j}(A) \dfrac{\nu_j}{\nu_{j+1}},\,~ \dfrac{\ga_H}{3}\right\},& \text% { if } \ell_j > \nu_j.
 %   
 % \end{array}  \right.
 \end{equation}
 and put $ \uli{\ga}_{\nu_{j+1}-1}(A) = \uli{\ga}_{\nu_{j+1}}(A)$. We will also define $\uli{\gamma}_n$ when $n$ is not a return time below.
 
\begin{Rem}
Note that $\uli{\gamma}_{\nu_j}$ assigns a number for each partition element. So, if $A_{\nu_j}$ and $A_{\nu_{j+1}}$ are respectively partition elements at time $\nu_j$ and $\nu_{j+1}$ and moreover $A_{\nu_{j+1}}\subset A_{\nu_j}$, we have $\uli{\gamma}_{\nu_j}(A_{\nu_{j+1}})=\uli{\gamma}_{\nu_j}(A_{\nu_j})$.
\end{Rem}

%Similarly, we define lower Lyapunov exponents $\uli{\ga}_{\nu_j}$ inductively as follows.
Suppose $\uli{\ga}_{k}(A)$ is defined for $\nu_0 \leq k \leq \nu_j$, that $\uli{\ga}_{\nu_j}(A) = \uli{\ga}_{\nu_j-1}(A)$, that (for $j \geq 1$)
\begin{equation}
 \ga_{k}(c) \geq \uli{\ga}_{k}(A) + (d-1)\al_{k+1}(c), \text{ for $\nu_{j-1} < k < \nu_j$}, \label{gamma-bar-ineq} 
\end{equation}
and that $\ga_{\nu_j}(c) \geq \uli{\ga}_{\nu_j}(A)$ at returns. 

If we follow a partition element $A$, and $c \in A$, we get, for those parameters not deleted, using that $|Df^{p_j+1}(\xi_{\nu_j}(c))| \geq e^{(1-\ka) r_j}$ by Corollary~\ref{bound-kappa}, 
\begin{align}
 &\ga_{\nu_{j+1}}(c) \geq \ga_{\nu_{j+1}-1}(c) - (d-1) \al_{\nu_{j+1}}(c) - \frac{C}{\nu_{j+1}} \\
 &\geq \left(\ga_{\nu_{j}-1}(c) (\nu_j-1) + (1-\ka) \al_{\nu_j}(c) \nu_j + \ga_H \ell_j \right) \frac{1}{\nu_{j+1}-1} \\
 &- (d-1)\ti{\al}_{\nu_{j+1}}(A) - \frac{C}{\nu_{j+1}} \\
 &\geq \left( \ga_{\nu_j-1}(A) \left( 1 - \frac{1}{\nu_j} \right) + (d-1)\al_{\nu_j}(c) \left( 1 - \frac{1}{\nu_j} \right)\right)\frac{\nu_j}{\nu_{j+1}-1} \\
 &+ \left( (1-\ka) \al_{\nu_j}(c) + \ga_H \frac{\ell_j}{\nu_j} \right) \frac{\nu_j}{\nu_{j+1}-1} 
 - (d-1) \frac{\ka'}{\ti{\ka}} \al_{\nu_j}(c) \frac{\nu_j}{\nu_{j+1}} - \frac{C}{\nu_{j+1}} \\
 &\geq \left( \ga_{\nu_j-1}(A) \left( 1 - \frac{1}{\nu_j} \right) + \ga_H \frac{\ell_j}{\nu_j} - \frac{C}{\nu_{j}} \right) \frac{\nu_j}{\nu_{j+1}-1} \\
&+ \left((d-1) \left( 1 - \frac{1}{\nu_j} - \frac{\ka'}{\ti{\ka}} \right) + (1-\ka)\right) \al_{\nu_j}(c) \frac{\nu_j}{\nu_{j+1}-1}. 
 \end{align}
 Note also that, since $\ka' < \ti{\ka}$, 
 \[
  \frac{\ga_{\nu_j-1}(A) + 2C}{\nu_j} \leq  \left( (d-1) \left( 1 - \frac{1}{\nu_j} - \frac{\ka'}{\tilde{\ka}} \right) + (1-\ka) \right) \al_{\nu_j}
 \]
 for sufficiently large $\nu_j \geq N$, since $\al_{\nu_j}(c) \nu_j \sim r_{\nu_j} \geq \De$ and $\De$ is large.
Then, using that $p_j \leq \nu_j$ by Lemma~\ref{bound-length}, we get $\ga_{\nu_{j+1}}(c) \geq \uli{\ga}_{\nu_{j+1}}(A)$. If we define $\uli{\ga}_k(A) = \uli{\ga}_{\nu_{j+1}}(A)$ for $\nu_{j} < k \leq \nu_{j+1}$, then it follows that 
 \begin{equation} \label{gamma-bar-ineq2}
  \ga_{k}(c) \geq \uli{\ga}_{k}(c) + (d-1) \al_{k+1}(c), \text{ for $\nu_j < k < \nu_{j+1}$.}
  \end{equation}
  So at returns we have the weaker inequality $\ga_{\nu_{j+1}}(c) \geq \uli{\ga}_{\nu_{j+1}}(A)$. With these definitions, we have that $\al_j(c) \leq \ti{\al}_{j}(A)$ and $\ga_k(c) \geq \uli{\ga}_k(c)$, with possible equality only at returns.
We also define $\uli{\ga}_j = \uli{\ga}$ for $j < N$. 
  We now make the following definition.

%\begin{Rem} Actually, we have something better, since $p < n$ by Lemma ---. Hence $\ga_{n+p+1} \geq \uli{\ga}/2$. By Lemma ---, we also have %$ \uli{\ga}_{\nu'} \geq \uli{\ga}_{\nu}/2$, where $\nu$ and $\nu'$ are consequtive free returns. If the return is inessential, we can choose %$\al_{\nu'}$ such that $e^{-\ka r} \sim e^{-r'} \sim e^{-\al_{\nu'} \nu'} $, where $r' = \al_{\nu'} \nu'$ which gives $\al_{\nu'} \leq \ka% %\al_{\nu}$. So in any case, $\al_{\nu'} \leq \ka' \al_{\nu}$.
%\end{Rem}

 \begin{Def}
  Parameters satisfying
  \[
   |\xi_j(c)| = e^{-\al_j(c) j},
  \]
  where $\al_j(c) \leq \ti{\al}_j(A)$ for all $c \in A=A_j(c)$ and for all $j \leq n$, are called parameters satisfying the {\em basic approach rate condition until time $n$}. We denote such parameters with $\BB_n$. 
  \end{Def}

 We will continue updating $\uli{\ga}_n$ and $\tilde{\alpha}_n$ as above until we reach a certain situation as described as follows. Let us now consider the consecutive free returns $N=\nu_0, \nu_1, \nu_2, \ldots$.
  
 We see that, if $A$ is a partition element at time $\nu_{j-1}$, 
 \[
  \frac{\al_{\nu_j}(c)}{\ga_{\nu_j}(c)} \leq 
\frac{\ti{\al}_{\nu_j}(A)}{\uli{\ga}_{\nu_j}(A)} \leq \ka'^j \frac{\ti{\al}_{\nu_0}(A)}{\uli{\ga}_{\nu_0}(A)}, \qquad \text{ where $c \in A$}. 
  \]
  Now we choose the smallest $J \geq 0$ so that, $\ti{\al}_{\nu_J} \leq \ti{C} \uli{\ga}_{\nu_J}$, where $\ti{C}$ is sufficiently small (depending on the starting function etc.), defined later. So $\nu_J$ is a time when the $\al_{\nu_J}(c)$ is (sufficiently) small compared to $\uli{\ga}_{\nu_J}$. The return times $\nu_J=\nu_j(c)$ depend on the parameter and are of course also constant on each partition element. We call $\nu_J = \nu_J(c)$ the end of the promotion, or the time period of the promotion for the corresponding parameter; the $\al_j$ becomes {\em promoted} in a sense to a small $\al=\al_{\nu_J}$.

It may seem that the same $\ka$ cannot be used during the promotion period, since the Lyapunov exponent after time $N$ may go down. However, since $1 > \ka' > \ka$ we keep only parameters that return further out from the critical point, i.e. if $\xi_{\nu_j}(A_j)$ is an essential return, and $A_j \subset \QQ$ is a partition element, then $\xi_{\nu_{j+1}}(A_j)$ either has reached the large scale before it returns (good situation), is an inessential return, or an essential return. If it is inessential or essential, the a partition element in $A_j$, let us call it $A_{j+1} \subset A_j$, after deletion by the basic approach rate condition, will satisfy
 \[
   |\xi_{\nu_{j+1}}(c)| \geq e^{-\ka' r_j} \gg e^{-r_j} \sim |\xi_{\nu_j}(c)|,
 \]
 for $c \in A_{j+1}$. Hence by Lemma~\ref{bound-order} the corresponding bound periods satisfy $p_{j+1} \leq p_j \leq \ldots \leq p_0 < N$ (the last inequality follows from Lemma~\ref{bound-length}). But this means that the corresponding $\ga_{p_j} \geq \uli{\ga}$, hence Corollary~\ref{bound-kappa} can be used during the whole promotion period for all returns $\nu_j$ up until $j=J$. For later use we put
 \[
  \ga_I = \min \{\uli{\ga}, \ga_H\} \frac{\uli{\ga} (1-\ka')}{2\al_{\nu_0}d + \uli{\ga} (1 -\ka')}. 
\]
 It will be shown that this is a lower Lyapunov exponent during the promotion period.  
  
 The next observation is that the measure of parameters deleted between two consecutive essential returns during the promotion period is exponentially small in terms of the return time of the former return. See Lemma~8.1 in~\cite{MA7}. 

 \begin{Lem} \label{basic-param}
  Let $\xi_{\nu}(A)$ be an essential return and $\nu$ belonging to the promotion period for $A \subset \EE_{\nu}(\ga_I) \cap \BB_{\nu}$ and let $\xi_{\nu'}(A)$ be the next essential return. Then if $\hat{A}$ is the set of parameters in $A$ that satisfy the basic approach rate condition ($\BB_{\nu'}$) at time $\nu'$, we have
  \[
m(\hat{A}) \geq \left(1-e^{-(3/2)(\ka' - \ka) \al_{\nu} \nu}\right)m(A). 
\]
\end{Lem}

%The proof is almost the same as in~\cite{MA7, ABC-1}. 
 \begin{proof}
We know that $\xi_{\nu}(A)$ grows rapidly during the bound period $p$. By Lemmas~\ref{bound-dist} and~\ref{md} and the definition of the bound period, we get, for any $c\in A$,
  \begin{equation}
   \diam(\xi_{\nu+p+1}(A)) \sim \frac{e^{-2r}}{r^{2}} \vert Df_{c}^p(\xi_{\nu+1}(c))\vert \geq e^{-\ka r},
  \end{equation}
  where $r = \al_{\nu} \nu$ and $\al_{\nu} = \al_{\nu}(A)$. 

If there are no inessential returns between $\nu$ and $\nu'$, then those parameters that get deleted satisfy
  $\al_{\nu'} = \al_{\nu'}(c) > \ti{\al}_{\nu'}(A) = \ka' (\nu/\nu') \al_{\nu}(A)$. 
  However, if we have, say, $s$ inessential returns between $\nu$ and $\nu'$, then we get $\al_{\nu'} > \ka'^{s+1} (\nu/\nu')\al_{\nu}$ (including the last return $\nu'$ which is essential). The growth between each pair of inessential returns $\xi_{m_j}(A)$ and $\xi_{m_{j+1}}(A)$ can be estimated by
  \[
|Df_{c}^{p_j}(\xi_{m_j}(c))| \geq e^{(1-\ka)r_j},
   \]
   where, by definition (see proof of Lemma~\ref{freein}), $\ka r_j \geq r_{j-1} $. So the total diameter becomes
   \begin{align}
    \diam(\xi_{\nu'}(A)) &\geq e^{-\ka r + (1-\ka)r_1 + (1-\ka)r_2 + \ldots + (1-\ka) r_s} \\
    &\geq e^{-\ka r + (1-\ka)r(\ka^{-1} + \ldots \ka^{-s})} = e^{ - r(\ka + 1 - \ka^{-s}) }.
    \end{align}

Note that by Lemma~\ref{freein}, no parameters are deleted during inessential returns. 
   
Recall that $r = \al_{\nu}\nu$. By the main distortion Lemma~\ref{md} together with Lemma~\ref{levin}, we see that the measure of parameters deleted at time $\nu'$ is 
  \begin{align}
   \frac{ m(A) - m(\hat{A})}{m(A)} &\leq C\, \frac{(2e^{-\al_{\nu'} \nu'})^2}{\diam(\xi_{\nu'}(A))^2} \\
   &\leq
   4Ce^{-2 \al_{\nu'} \nu' + 2 (\ka + 1 - \ka^{-s}) \al_{\nu} \nu} \leq 4C e^{-2(\ka'^{s+1} - 1- \ka +\ka^{-s} ) \al_{\nu} \nu},
  \end{align}
 where $C$ is a distortion constant (close to $1$). Since $\kappa'^{s+1} + \kappa^{-s}$ increases in $s \geq 0$ we have $\ka'^{s+1} - 1 - \ka + \ka^{-s} \geq \ka'-\ka > 0$ for all $s \geq 0$. This finishes the proof since $\al_{\nu} \nu \sim r$ is large (so that the constant $4C$ in front is eaten up). 
 \end{proof}
 We end this section by summarising the results above. We are going to continue the promotion until $\al_{\nu_j}$ is small enough, but not too small. With a prescribed $\al=\al_{\nu_J} > 0$ (defined later) we get the following, if $\al_{\nu_j} \geq \al$. 
  \begin{Prop}
   Let $\nu_0=N$ as before. Then during the promotion period, from $\nu_0$ until $\nu_J$, the set of parameters that satisfy $\BB_{\nu_J}$, has measure at least
   \[
 \left[\prod_{j=0}^{J-1}\left(1-e^{-(3/2)(\ka'-\ka) \al_{\nu_j} \nu_j}\right)\right] (1-\delta_0)m(\QQ) = (1-\de_1)(1-\delta_0) m(\QQ).
 \]
   \end{Prop}

   So the remaining parameters constitute a finite union of squares, were the $\al=\al_{\nu_J}(c)$ are sufficiently small in order to use the large deviation argument. Note that $\de_1$, the portion deleted, can be made arbitrarily small by decreasing the size of $\QQ$. Heuristically, the method of promotion indicates that all CE-maps are density points of slowly recurrent maps, since one can choose $\al > 0$ arbitrarily small (possibly diminishing the diameter of $\QQ$).

\section{After promotion}

We are going to consider iterates from some starting time $m_0=\nu_J$, which is the end of the promotion period (it could be $N$, but not necessarily) up to $(1+\io)m_0$ for some suitable $\io > 0$.

\subsection{Uniform bound of the promotion period} \label{large-dev}

In this section we choose $\ti{C}$ so that $\al_{\nu_J}/\uli{\ga}_{\nu_J}$ is small enough. Let us see how the Lyapunov exponent changes during the time of promotion. During the promotion period between $N=\nu_0$ and $\nu_J=m_0$ the Lyapunov exponent cannot drop too much, since all the bound periods $p_j$ are bounded by $p_0$, by Lemma~\ref{bound-order}. But we have something stronger. Recall that $\uli{\ga}_j$ is also defined for $j < N$, namely $\uli{\ga}_j = \uli{\ga}$, and $ \ga_{p_j} \geq (d-1) \al_{p_j+1} + \uli{\ga}_{p_j}$, where $\uli{\ga}_{p_j}=\uli{\ga}$ since $p_j \leq p_0 < N$. By Lemma~\ref{bound-exp}, we have, if $0 < \eta < 1$,
  \[
   p_{j+1} \leq \frac{(1 + \eta) dr_{j+1}}{\ga_{p_{j+1}-1} + \al_{p_{j+1}} + \be} \leq \frac{2d \ka' r_j}{\uli{\ga}} \leq \frac{2d \ka'^j r_0}{\uli{\ga}}.
  \]
  So the sum of all bound periods during the promotion time is bounded by
  \[
\sum_{j=0}^J p_j \leq \sum_{j=0}^J \frac{2d \ka'^j r_0}{\uli{\ga}} \leq \frac{2d r_0}{\uli{\ga}} \frac{1}{1- \ka'}.
   \]

If we set the Lyapunov exponent to be zero at the bound periods in the promotion period, and count only the free periods $\ell_j$ after $\nu_0$, we get the following estimate of the Lyapunov exponent after the promotion period, if $\sum_j \ell_j$ is all the free periods after $\nu_0$: 
\begin{equation} \label{gammaI}
\ga_{\nu_J} \geq \frac{\uli{\ga} \nu_0 + \ga_H \sum_j \ell_j}{\frac{2d \al_{\nu_0}}{\uli{\ga}} \frac{1}{1- \ka'} \nu_0 + \nu_0 + \sum_j \ell_j} \geq \min\{\uli{\ga},\ga_H\} \frac{\uli{\ga} (1-\ka')}{2 \al_{\nu_0} d + \uli{\ga} (1- \ka')} = \ga_I.
\end{equation}

\begin{Rem}
According to Remark~\ref{p-not-return}, $p_j$ cannot be a return time, but $p_j+1$ can. Moreover, if the bound period ends so that $p_j+1$ is not a return time, then $\xi_{\nu_j+p_j}(A_{\nu_j})$ has grown to size $\sim e^{-\be p_j} e^{-\al_{p_j}p_j}$. But since $\xi_{\nu_j+p_j}(A_{\nu_j}) \cap U = \emptyset$, we have $\al_{p_j} \approx 0$ and since $r_j = \al_{\nu_j}\nu_j$,
   \begin{align}
    \diam(\xi_{\nu_j+p_j}(A_{\nu_j})) &\sim e^{-\be p_j} \sim e^{-\frac{2d \be}{\ga_{p_j-1}+\be} r_j} \nonumber \\
                  &= e^{-\frac{2d \be \al_{\nu_j}}{\ga_{p_j-1}+\be} \nu_j} \geq e^{-\frac{2d \be \al_{\nu_j}}{\ga_{p_j-1}+\be} (\nu_j+p_j)} \geq
                   e^{-\frac{2d \be \al_{\nu_j}}{\uli{\ga}}(\nu_j+p_j)}.
    \end{align}
    So for a prescribed $\al > 0$ which we are looking for after promotion has ended, if $\be > 0$ is chosen sufficiently small, then $\xi_{\nu_j+p_j}(A_{\nu_j})$ has already reached a large size and the promotion period has ended. We could choose $\be$ so that, for instance, 
    \[
     \frac{\al}{2} > \frac{2d \be \al_{\nu_0}}{\uli{\ga}} \geq \frac{2d \be \al_{\nu_j}}{\uli{\ga}}.
     \]
     This means that we may assume that during the promotion period, all essential return occur directly after the bound returns. In particular, we may assume that $\sum_j \ell_j = 0$ in~\eqref{gammaI}. 
\end{Rem}

We may assume without loss of generality that $\ga_H \geq \uli{\ga}$. Note also that $\ga_I$ only depends on the starting map and $\ka'$, since $\al_{\nu_0} < \oli{\ga} - \uli{\ga}$. Set
 \[
\hat{\ka} = 1 -\frac{\ga_{I}}{4 \oli{\ga}}.
\]
This is an updated $\ka$, since we cannot guarantee that $\ga \geq \uli{\ga}$ anymore. We want to choose $\al=\al_{\nu_J}$ so small such that
\begin{equation} \label{alpha}
 0 < \frac{\al}{\ga_I} \leq \min\left\{\frac{1}{16d},\, \frac{\ga_C^3}{1000 d \oli{\ga}^2 }\right\},
\end{equation}
where $\ga_C = \min\{\ga_I/(12d), \ga_H \}$.

\subsection{Large deviations}

As in previous papers,~\cite{BC1, BC2} and~\cite{ABC-1} we define escape time as follows. 

\begin{Def}
We say that $\xi_{n}(A_{n-1})$ has escaped or is in escape position, if $n$ does not belong to a bound period and $\diam(\xi_{n}(A_{n-1})) \geq S$. We also speak of escape situation for $A_{n-1}$ and say that $A_{n-1}$ has escaped if $\xi_{n}(A_{n-1}) $ has escaped. 
\end{Def}

Escape situations are the ideal situations since it means that a partition element has grown to the large scale under $\xi_n$. If we are lucky, even $\QQ$ grows to the large scale, and then $\xi_n(\QQ)$ covers mostly the Fatou set of $f_0$, and by a compactness argument, we can conclude that most parameters in $\QQ$ will land in the Fatou set of the corresponding map under $\xi_n$. Hence a very large portion of parameters are hyperbolic in $\QQ$. If $\xi_n(\QQ)$ does not escape, we have to continue with the ordinary partitions. Also since the starting map is assumed to not be slowly recurrent, we want to delete parameters so that the remaining parameters are ``promoted'' in a way that after a finite number of free returns, these parameters have reached a kind of slow recurrence condition.

In the following lemma, we estimate the maximum escape time for essential returns after promotion in an interval $[m_0,2m_0]$. Note that during this period, the exponent cannot drop below $\ga_I/4$. 
\begin{Lem}[q-Lemma]
 Given an essential return $\xi_{\nu}(A) \subset U$, where $A \subset \EE_{\nu}(\ga_I/4) \cap \BB_{\nu}$, $\diam(\xi_{\nu}(A)) \sim e^{-r}$ and $\nu \in [m_0,2m_0]$. Let $n$ be the smallest number so that either $\xi_n(A)$ is an essential return or has escaped. Then, for $q = n-\nu$, we have
 \[
q \leq \ti{M} r,
  \]
where $\ti{M}$ is a constant only depending on the unperturbed map and $\ka'$. 
\end{Lem}

\begin{proof}
 % More or less the same as in~\cite{MA7}.
 
If the bound period $p \leq \nu_J=m_0$, we have, by the definition of the bound period, replacing $\uli{\ga}$ with $\ga_I$ in the proof of Corollary~\ref{bound-kappa} (see also Lemma~\ref{bound-exp}), for all $c \in A$, 
 \begin{equation} \label{b-exp}
  \left|Df_c^{p+1}(\xi_{\nu}(c))\right| \geq e^{(1-\hat{\ka}) r}.
\end{equation}
Let $m_j$ be the inessential returns after $\nu$, i.e. $\nu < m_1 < m_2 < \ldots < m_s < n$. Let $p_j$ and $q_j$ be the bound and free periods respectively following $m_j$. Let $p_0$ and $q_0$ be the bound and free periods following the return $\nu$. It can happen that escape takes place before a return takes place, and then $q_s$ is not a complete free period. It can also happen that $n$ is a time during the bound period for $m_s$. But then we have $m_s+p_s - \nu$ as an upper bound for $q$ and we can assume that $q > m_s+p_s$.

Suppose that $\dist(0, \xi_{m_j}(A)) \sim e^{-r_j}$, and let $r=r_0$. Suppose that $n=\nu'$ is a return. As long as the bound period is bounded by $(8 d \al /\ga_I)\nu \leq \nu/2 \leq m_0$, since $\al = \al_{\nu_J}$ and $\al/\uli{\ga}_I \leq 1/(16d)$, we can use the same estimate as~\eqref{b-exp}, the Main Distortion Lemma~\ref{md} and Lemma~\ref{oel}, to obtain
\begin{align}
 \diam(\xi_{n}(A)) &\sim |Df_c^{n - \nu}(\xi_{\nu}(c))| \diam(\xi_{\nu}(A)) \nonumber \\
 &=
  \prod_{j=0}^{s} |Df_c^{p_j}(\xi_{m_j}(c))| C' e^{\ga_H q_j} \diam(\xi_{\nu}(A)) \nonumber \\
 &\geq
  e^{\ka r} C' e^{\ga_H q_0} \diam(\xi_{\nu}(A)) \prod_{j=1}^s e^{r_j(1-\hat{\ka})} \prod_{j=1}^s C' e^{q_j \ga_H}.  \label{last}
\end{align}

If $n$ was not a return, then let $q_1 < \ldots < q_t$ be the pseudo-returns after $m_s+p_s$. Between each pair of pseudo-returns we have uniform expansion of the derivative according to Lemma~\ref{adjacent-returns}. Between $m_s+p_s$ and $q_1$ we also have uniform expansion according to Lemma~\ref{oel}. So we only need to consider the last time period, from $q_t$ to $n$. Since $\xi_{q_t}(A)$ may belong to $U' \sm U$ we have $|Df_c(\xi_{q_t}(c))| \geq e^{-(d-1) \De}$ for all $c \in A$. After time $q_t$ we can use the binding information, Lemma~\ref{bound-dist} and the first statement of Lemma~\ref{oel} with $U = U'$, depending on whether $n$ belongs to the bound period or not. In any case we get uniform expansion; $|Df_c^{n-q_t-1}(\xi_{q_t+1}(c))| \geq C e^{\ga_H (n-q_t-1)}$, where $C$ does not depend on $\de$ (but possibly $\de'$). In other words, with $z = \xi_{m_s+p_s}(c)$, for $c \in A$, we have
\begin{align}
 |Df_c^{n-(m_s+p_s)}(z)| &= |Df_c^{q_1-(m_s+p_s)}(z)| |Df_c^{q_2-q_1}(f_c^{q_1-(m_s+p_s)}(z))| \nonumber \\
  &\cdot \ldots \cdot |Df_c^{q_t-q_{t-1}}(f_c^{q_{t-1}-(m_s+p_s)}(z))| |Df_c(f_c^{q_t-(m_s+p_s)}(z))| \nonumber \\
 &\cdot |Df_c^{n-q_t-1}(f_c^{q_t-(m_s+p_s) + 1}(z))| \nonumber \\
             &\geq C' e^{\ga_H (q_1 -(m_s+p_s))} e^{\ga_2 (q_t-q_1)} e^{-(d-1) \De} C e^{\ga_H (n-q_t-1)} \nonumber \\
              &\geq e^{\ga_C (n-(m_s+p_s))} e^{-(d-1)\De},
\end{align}
where $\ga_2 \geq \ga_I /(12d) \geq \ga_C$.
%We also assume that $\ga_I \leq \ga_H$, otherwise take $\ga_I = \ga_H$. 
So we may have to replace $q_s \ga_H$ with $q_s \ga_C- (d-1) \De$ in~\eqref{last}, where $q_s = n-(m_s+p_s)$ in this case. 
Since $\ga_C \leq \ga_H$, and $\diam\xi_{n}(A)$ is assumed to be at most $S = \vep_1 \de \leq 1$, we therefore get
\[
\sum_{j=1}^s r_j (1-\hat{\ka}) + \sum_{j=0}^s q_j \ga_C \leq r \hat{\ka} + (d-1) \De. 
\]
Hence, if $q = \sum_{j=0}^s p_j + \sum_{j=0}^s q_j$, we get, by Lemma~\ref{bound-exp} ($p \leq 2dr /(\ga_{p-1} + \al_p + \be)$), 
\begin{align}
 q - p_0 &= \sum_{j=1}^s p_j + \sum_{j=0}^s q_j \leq \sum_{j=1}^s \frac{(1+\eta)d}{\ga_{p_j - 1} + \al_{p_j} + \be} r_j + \sum_{j=0}^s q_j \nonumber \\
  &\leq \sum_{j=1}^s \frac{(1+\eta)d}{\uli{\ga}_{\nu_J} (1-\hat{\ka})} (1 - \hat{\ka}) r_j + \frac{1}{\ga_C} \sum_{j=0}^s q_j \ga_C \nonumber \\
 &\leq \max \biggl\{ \frac{(1+\eta)}{(1-\hat{\ka})\uli{\ga}_{\nu_J}}, \frac{1}{\ga_C} \biggr\}
 \biggl( \sum_{j=1}^s r_j (1-\hat{\ka} ) + \sum_{j=0}^s q_j\ga_C \biggr) \nonumber \\
  &\leq \max \biggl\{ \frac{(1+\eta)}{(1-\hat{\ka}) \uli{\ga}_{\nu_J}}, \frac{1}{\ga_C} \biggr\} \biggl( \hat{\ka} r +(d-1) \De \biggr)\\ &\leq M r,
\end{align}
where $M = 2d/ (\ga_C(1-\hat{\ka}))$. Moreover, $p_0 \leq 8dr/\ga_I$ (since the exponent $\ga \geq \ga_I/4)$, so
\[
q \leq Mr + \frac{2d}{\ga_I}r \leq \frac{4d}{(1-\hat{\ka})\ga_C}r \leq \frac{16 d \oli{\ga}}{\ga_C^2}r.
\]
From~\eqref{gammaI} and since $(d-1) \al_{\nu_0} < \oli{\ga} - \uli{\ga}$, we see that $\ga_I \geq C_I \uli{\ga}$, where $C_I$ only depends on the starting map and $\ka'$. Hence $\ga_C$ only depends on the starting map. We now put
\[
 \ti{M} = \frac{16 d \oli{\ga}}{\ga_C^2}.
 \]
With $\al = \al_{\nu_J}$, and the conditions on $\al$ (see~\eqref{alpha}) we obtain 
\[
\ti{M} \al \leq \frac{1}{2}. 
 \]
Since the total time is bounded from above by $\nu + \ti{M} r \leq \nu + \ti{M} \al_{\nu} \nu \leq (3/2) \nu$, we can use the binding information the whole time.
\end{proof}

We will now estimate the set of parameters that do not escape after a long time.
\begin{Def}
A partition element $A$ at time $n$ that contains the parameter $c$ is denoted by $A_n(c)$. Given an essential return $\xi_{\nu}(A)$, and $c \in A=A_{\nu}(c)$ we define the {\em escape time} $E(\nu,c)$ to be the time $n-\nu$ where $\xi_{n}(A_{n-1}(c))$ is in escape position for the first time after $\nu$. 
\end{Def}

By Lemma~\ref{startup}, we can assume without loss of generality that $\xi_N(\QQ)$ is an essential return. Otherwise, if $\xi_N(\QQ)$ is in escape position, we can go directly to the next section. The point is that most parameters in $\QQ$ escape before a certain time after $N$. Suppose $m_0 = \nu_J$ is an essential return time for a partition element $A=A_0$. We now want to choose $\io$ so that the Lyapunov exponent cannot drop under $\ga_C$ anytime in the time window $[m_0,(1+\io)m_0]$.
%During the bound period for a return $\xi_{\nu}(A) \subset U$, where $|\xi_{\nu}(a)| \sim e^{-r}$ for $a \in A$, we have %$|Df^p(\xi_{\nu}(a))| \geq e^{(1-\ka)r}$, and we can use the trivial bound $|Df^p(\xi_{\nu}(a))| \geq 1$.

Suppose that $\nu_J = m_0 < m_1 < \ldots < m_s < (1+\io) m_0$ are consecutive essential returns for a parameter $a \in \QQ$. Let $A_j$ be a partition element for an essential return $\xi_{m_j}(A_j) \subset U$ where $\dist(0,\xi_{m_j}(A_j)) \sim e^{-r_j}$. Then, by Lemma~\ref{bound-exp}, 
\[
\frac{m(A_{j+1})}{m(A_j)} \leq C \frac{e^{-2r_{j+1}}}{e^{-2\hat{\ka} r_j}}.
 \]
 Consequently, 
 \[
\frac{m(A_s)}{m(A_0)} \leq C^s \frac{e^{2\sum_{j=0}^{s-1} -r_{j+1}}}{e^{-2\hat{\ka} \sum_{j=0}^{s-1} r_j}} \leq C^s e^{2\hat{\ka} r_0 - 2(1-\hat{\ka}) \sum_{j=1}^s r_j}.
  \]
  Putting $R = r_1 + \ldots + r_s$ we have, by the q-Lemma,
  \[
E(m_0,c) \leq \sum_{j=0}^s \ti{M} r_j \leq \ti{M} r_0 + \ti{M} R.
   \]

   Following the old traditions (the ideas go back to the works~\cite{BC1, BC2}), we now compute the number of combinations of returns, i.e. the number of combinations of $R = r_1 + \ldots + r_s$ where $R$ is fixed. This can be estimated with Stirling's formula: If $\De$ is sufficiently large, then
\[
\binom{R+s-1}{s-1} \leq e^{R (1-\hat{\ka})/3} (1 + \eta(\De))^R,
\]
where $\eta(\De) \raw 0$ as $\De \raw \infty$. Let $A_{s,R}$ be the set of parameters in $A_0$ that have exactly $s$ free essential returns in the time interval $[m_0,(1+\io)m_0]$ before escape, where $R=r_1 + \ldots + r_s$. Let $\hat{A}$ be the largest of all squares in $A_{s,R}$. Then,
\[
m(A_{s,R}) \leq m(\hat{A}) e^{R(1-\hat{\ka})/3} (1+\eta(\De))^R.
\]
So,
\begin{align}
 &m\left(\left\{ c \in A_0 : E(m_0,c) = t \right\}\right) \\
 &\leq \sum_{R \geq \frac{t}{\ti{M}}-r, s \leq R/\De} m(A_{s,R}) \\
 &\leq \sum_{R \geq \frac{t}{\ti{M}}-r} \sum_{s = 1}^{R/\De} m(\hat{A}) e^{R(1-\hat{\ka})/3} (1+\eta(\De))^R\\
                  &\leq m(A_0) \sum_{R \geq \frac{t}{\ti{M}}-r} C' e^{R(1-\hat{\ka})/3} (1+\eta(\De))^R C^{R/\De}e^{2 \hat{\ka} r - 2(1-\hat{\ka})R} \\
                  &\leq C m(A_0) \sum_{R \geq \frac{t}{\ti{M}} - r} e^{R(1-\hat{\ka})/3 + R \ln (1 + \eta(\De)) + (R/\De) \log C + 2\hat{\ka} r - 2(1-\hat{\ka}) R} \\
                  &\leq C m(A_0) \sum_{R \geq \frac{t}{\ti{M}}-r} e^{-\frac{3}{2}(1-\hat{\ka}) R + 2\hat{\ka} r} \\
                  &\leq C m(A_0) e^{-\frac{3}{2}(1-\hat{\ka})(\frac{t}{\ti{M}}-r) + 2\hat{\ka} r} \\
 &= C m(A_0) e^{-\frac{3}{2}(1-\hat{\ka}) \frac{t}{\ti{M}} + \frac{r}{2}(3+\hat{\ka})} .
 \end{align}
 So if $t \geq 2\ti{M} \frac{3+\hat{\ka}}{1-\hat{\ka}} r$ then we have an estimate of the measure for parameters with long escape times (recall $r = \al m_0$):
 \[
m\left(\{ c \in A_0 : E(m_0,c) = t \}\right) \leq C m(A_0) e^{-2(3+\hat{\ka})r} = C m(A_0) e^{-2 (3+\hat{\ka}) \al m_0}, 
\]
or
\[
m(\{ c \in A_0 : E(m_0,c) = t \}) \leq C m(A_0) e^{- (1-\hat{\ka}) \frac{t}{\ti{M}}}.
 \]

 This means that a big portion of $A_0$ will reach the large scale within the time period $[m_0, (1+\io)m_0]$, where $\io = 10 \ti{M} \frac{3+\hat{\ka}}{1-\hat{\ka}} \al \leq 1$ (see the condition on $\al$ in~\eqref{alpha}). Let this portion be $1-\de_2$, where $0 < \de_2 < 1$, is small. Again note that $\de_2$ can be chosen as close to zero as we want.

 Since this is true for every such $A_0 \subset \QQ$, we have that the measure of parameters that reach the large scale within time $(1+\io)m_0$ is at least
 \[
m(\QQ)(1-\delta_0) (1-\de_1)(1-\de_2).
  \]

    From $m_0$ to $m_0(1+\io)$, we do not loose too much either, since the Lyapunov exponent during this time cannot drop below $\ga_I/(2(1+\io)) \geq \ga_C$. Hence we have good distortion control for all parameters satisfying the basic approach rate assumption until $(1+\io)m_0$.

\section{Hyperbolic maps}

Now to prove our theorems, we can argue similarly as in~\cite{ABC-1}. Let $\FF(f_{c_0})$ be the Fatou set for the unperturbed map $f_{c_0}$ which is Collet--Eckmann, and let $\NN_{\vep_0}$ be an $\vep_0$-neighbourhood of $\JJ(f_{c_0})$ as defined in~\eqref{nbjulia}. Here, $\vep_0$ can be chosen arbitrarily close to zero since the Hausdorff dimension of the Julia set of $f_{c_0}$ is strictly less than two and therefore has zero area~\cite[Corollary 2]{GS}. Then since the Fatou set of $f_{c_0}$ consists only of the attracting basin of $\infty$, any compact subset of the Fatou set is stable under perturbation. So $\widehat{\C} \sm \NN_{\vep_0}$ is still contained the Fatou set $\FF(f_c)$ for all $c \in \QQ$, provided that $\QQ$ is chosen small enough. By bounded distortion
 \begin{equation}\label{fulldense}
m\left(\{ c \in \QQ: c \in \FF(f_c) \}\right) \geq (1-\de_0)(1-\de_1)(1-\de_2)(1-\de_3) m(\QQ),
  \end{equation}
where $\de_3$ is a constant depending on the distortion of $\xi_n(c)$ on $\QQ$. Since also $0 < \de_3 < 1$ can be arbitrarily small, it follows that the Lebesgue density of hyperbolic maps at $c=c_0$ is equal to $1$, by making the diameter of $\QQ$ tending to zero.

To see that Theorem~\ref{mainthm} actually holds, we note that for all parameter $c\in \QQ$ satisfying $c\in \FF(f_c)$, the critical point is attracted by the super-attracting fixed point at $\infty$. Therefore,~\eqref{fulldense} actually gives an estimate of parameters which lie outside of $\MM_d$.
%Theorem~\ref{mainthm} follows directly by making $\varepsilon\to 0$ (the sidelength of $\QQ$).

\begin{Rem}
For Collet--Eckmann rational maps with precisely one critical point in the Julia set, our proof works completely analogously. In this case, the parameter space is higher dimensional. Instead of using the transversality result by Avila and also Levin mentioned before, we shall need a more general result by Levin~\cite{Levin-transv}. To get Lebesgue density of hyperbolic maps, we consider one-dimensional slices in the parameter space and get the same estimate as above.
\end{Rem}

\begin{Rem}
In the setting of rational maps with precisely one critical point in the Julia set, the Collet--Eckmann condition is equivalent to several other well known non-uniform hyperbolicity conditions, e.g., the topological Collet--Eckmann condition, the second Collet--Eckmann condition, uniform hyperbolicity on periodic orbits, etc, see~\cite{PRS}. So our result applies to these maps as well. 
\end{Rem}

\begin{Rem}
 The methods of promotion and the following parameter exclusion are likely to generalise to all rational CE-maps. Since the recurrence exponent $\al > 0$ after the promotion period can be chosen arbitrarily small (with possibly diminishing the diameter of $\QQ$), taking into consideration that there may be several critical points in the Julia set (see~\cite{MA7},~\cite{ABC-1}), one could prove that every rational CE-map for which $J(f) = \hat{\C}$ is a Lebesgue density point of slowly recurrent CE-maps, even for real-analytic families of maps containing the starting CE-map (cf.~\cite{Lefevre-thesis} where complex one-dimensional families were considered). 
\end{Rem}

\bibliographystyle{plain}
\bibliography{ref}

\bigskip

\noindent {\bf Magnus Aspenberg}\\
Centre for Mathematical Sciences\\
Lund University, Box 118, 22 100 Lund, Sweden

\smallskip
 
\noindent{magnus.aspenberg@math.lth.se}

\bigskip

\noindent {\bf Mats Bylund}\\
Laboratoire de Math\'ematiques d'Orsay\\
Universit\'e Paris--Saclay, 91405 Orsay Cedex, France

\smallskip

\noindent{mats.bylund@universite-paris-saclay.fr}

\bigskip

\noindent {\bf Weiwei Cui}\\
 Research Center for Mathematics and Interdisciplinary Sciences\\
Frontiers Science Center for Nonlinear Expectations, Ministry of Education\\
Shandong University, Qingdao, 266237, China.

\smallskip

\noindent{weiwei.cui@sdu.edu.cn}

\end{document}